\documentclass{article}

 \usepackage[final]{neurips_2025}

\usepackage[utf8]{inputenc} %
\usepackage[T1]{fontenc}    %
\usepackage{hyperref}       %
\usepackage{url}            %
\usepackage{booktabs}       %
\usepackage{amsfonts}       %
\usepackage{nicefrac}       %
\usepackage{microtype}      %
\usepackage{xcolor}         %

\usepackage{amsmath,amssymb,amsthm}

\usepackage{subcaption}
\usepackage{enumitem}
\setlist{noitemsep,topsep=0pt,parsep=0pt,partopsep=0pt,leftmargin=*}

\usepackage{todonotes}

\usepackage[suppress]{color-edits}
\addauthor{gb}{blue}
\addauthor{pbl}{brown}
\addauthor{len}{red}

\newcommand{\mathsep}{,~}

\newcommand{\set}[1]{\left\lbrace #1 \right\rbrace}
\newcommand{\card}[1]{\left\lvert{#1}\right\rvert}

\newcommand{\norm}[2]{\left\lVert{#1}\right\rVert_{#2}}

\newcommand{\setR}{\mathbb R}
\newcommand{\setN}{\mathbb N}

\DeclareMathOperator*{\argmin}{arg\,min}

\DeclareMathOperator*{\append}{\textsc{append}}
\DeclareMathOperator*{\update}{\textsc{update}}

\DeclareMathOperator*{\shuffle}{\textsc{shuffle}}
\DeclareMathOperator*{\exchange}{\textsc{exchange}}

\newtheorem{definition}{Definition}
\newtheorem{proposition}{Proposition}
\newtheorem{theorem}{Theorem}
\newtheorem{lemma}{Lemma}

\newtheorem{remark}{Remark}

\newcommand{\ALG}{\textsc{Alg}}

\newcommand{\GBT}{\textsc{GBT}}

\newcommand{\cA}{\mathcal A}

\newcommand{\cD}{\mathcal D}
\newcommand{\cE}{\mathcal E}

\newcommand{\cL}{\mathcal L}

\newcommand{\bD}{{\bf D}}

\newcommand{\bbR}{\mathbb R}

\newcommand{\eqdef}{\triangleq}

\newcommand{\debug}[1]{#1}		%

\newcommand{\newop}[2]{\DeclareMathOperator{#1}{\debug{#2}}}		%

\newop{\reg}{\mathcal{R}}
\newop{\setComp}{\mathfrak{R}}

\newcommand{\ab}{a\ominus{} b}
\newcommand{\anbn}{a_n\ominus{} b_n}

\title{Generalizing while preserving monotonicity \\ in comparison-based preference learning models}

\author{%
  Julien Fageot\thanks{Equal contribution.} \\
  Tournesol \\
  \And
  Peva Blanchard${}^{*}$ \\
  Kleis Technology \\
  \And
  Gilles Bareilles \\
  CTU in Prague \\
  \And
  Lê-Nguyên Hoang \\
  Calicarpa, Tournesol \\
}

\begin{document}
\maketitle

\begin{abstract}
    If you tell a learning model 
    that you prefer an alternative $a$ over another alternative $b$,
    then you probably expect the model to be \emph{monotone}, that is, the valuation of $a$ increases, 
    and that of $b$ decreases.
    Yet, perhaps surprisingly, many widely deployed comparison-based preference learning models, including large language models, fail to have this guarantee.
    Until now, the only comparison-based preference learning algorithms 
    that were proved to be monotone are the Generalized Bradley-Terry models~\cite{DBLP:conf/aaai/FageotFHV24}.
    Yet, these models are unable to generalize to uncompared data.
    In this paper, we advance the understanding of the set of models
    with generalization ability that are \emph{monotone}.
    Namely, we propose a new class of Linear Generalized Bradley-Terry models with Diffusion Priors,
    and identify sufficient conditions on alternatives' embeddings 
    that guarantee monotonicity.
    Our experiments show that this monotonicity is far from being a general guarantee,
    and that our new class of generalizing models improves accuracy,
    especially when the dataset is limited.
\end{abstract}

\section{Introduction}
\label{sec:intro}

Preference learning, sometimes known as \emph{alignment}, has become central to machine learning.
In particular, in recent years, there has been a growing interest to leverage comparative judgments to fine-tune AI models, using frameworks like \emph{Reinforcement Learning with Human Feedback} (RLHF)~\cite{DBLP:conf/nips/ChristianoLBMLA17} or \emph{Direct Preference Optimization} (DPO)~\cite{rafailovDirectPreferenceOptimization2024} in the context of language models, or linear models in the context of applications ranging from trolley dilemmas to food donation~\cite{awadMoralMachineExperiment2018,DBLP:journals/pacmhci/LeeKKKYCSNLPP19}.
These models are now deployed at scale.
\gbedit{In parallel, learning preferences from comparisons with mathematical guarantees also fits in social choice theory, contributing to developing more transparent social medial, as advocated recently by ``prosocial media'' \cite{weyl2025prosocialmedia}, with a direct application for collaborative scoring of social media content \cite{hoang2021tournesolquestlargesecure,hoang2024solidagomodularcollaborativescoring}, consensus-driven polling \cite{small2021polis}, or recommender system based on explicit preference such as~\cite{flicke2025scholarinboxpersonalizedpaper}, among others.}

Yet, bizarre aspects of these preference learning algorithms are regularly observed.
\gbreplace{In particular, strikingly, \cite{DBLP:journals/corr/abs-2402-13228,DBLP:journals/corr/abs-2410-08847,bareillesMonotonicityAIAlignment2025} showed that, when they are updated based on a comparative judgment saying that an alternative $a$ is preferable to $b$, the probability of generating item $a$ can in fact decrease.}{A common but striking observation is that, when updating a model based on a comparison judging that an item $a$ is preferable to an item $b$, the probability of $a$ can decrease  \cite{DBLP:journals/corr/abs-2402-13228,DBLP:journals/corr/abs-2410-08847,bareillesMonotonicityAIAlignment2025}.}
In fact, most deployed learning algorithms fail to guarantee \emph{monotonicity}: the probability or score of an item may be reduced after it was said to be better than another item.

Perhaps surprisingly, 
the root cause of this lack of monotonicity is \emph{not} the nonlinearity of the models.
In fact, we can expose this ``bug" with a very basic example.
Consider a linear two-dimensional model with a parameter $\beta \in \setR^2$ to be learned: the scores of items $a$ and $b$ are $\beta^\top x_a$ and $\beta^\top x_b$, where $x_a$ and $x_b$ are (two-dimensional) vector embeddings.
We are given a comparison that favors alternative $a$ over $b$, whose embeddings are $x_a = (1, 0)$ and $x_b = (2, 0)$.
Since $x_{a1} < x_{b1}$, this comparison will push $\beta_1$ towards lower values.
But since the score of $a$ according to the linear model is $\beta^T x_a = \beta_1$,
this means that the score of $a$ will also decrease.
Thus, including a comparison that favors $a$ over $b$ has decreased the score of $a$.
This becomes all the more troubling when there exists an alternative $c$ with embedding $x_c = (0, 1)$, whose score remains unchanged.

This example questions whether preference learning algorithms can be trusted.
A user who witnesses a surprising evolution of alternatives' scores as illustrated above might, understandably, prefer not to use such an algorithm.
Worse yet, if they are nevertheless forced to use the algorithm, they could want to remove the data they previously provided, because the eventual learned model was deteriorated by their truthful data reporting.
More generally, this may discourage users to report their preferences, and rather provide tactical preferences in the hope to steer the model towards their goal.
This is reminiscent of tactical or ``useful'' strategies in voting systems.
\gbreplace{Now, some models, like the (Generalized) Bradley-Terry model, were proved to be monotone~\cite{DBLP:conf/aaai/FageotFHV24}.}{There exist classes of models which have a mathematical guarantee of monotonicity, such as (Generalized) Bradley-Terry models~\cite{DBLP:conf/aaai/FageotFHV24}; see also \cite{DBLP:conf/nips/NoothigattuPP20}.}
\gbreplace{But, to the best of our knowledge, none of them enables score generalization to non-evaluated alternatives.}{To the best of our knowledge, existing models with a guarantee of monotonicty fail to generalize: they cannot predict the score of items that have not been compared.}
Hence the following research problem.
\begin{center}
    \emph{Can a generalizing comparison-based preference learning algorithm guarantee monotonicity?}
\end{center}

\paragraph{Contributions.}
Our first contribution is to identify 
a large class of preference learning algorithms that leverage both \emph{comparisons} between alternative pairs, and \emph{descriptive information} (embeddings) on individual alternatives,
which we call \emph{Linear Generalized Bradley-Terry with Diffusion Prior} (Definition \ref{def:gbt-with-prior-similarities}).
This class extends the (Generalized) Bradley-Terry models \cite{DBLP:conf/aaai/FageotFHV24}
by including a linear mapping of the embeddings---and priors on alternative similarities,
thereby allowing preference \emph{generalization} to yet uncompared alternatives.

As a second contribution, we provide conditions on the embeddings that guarantee that the learning algorithm behaves monotonically when new comparisons are provided.
As discussed above, this property is highly desirable and yet hard to guarantee in practice.
In particular, we propose a class of \emph{diffusion embeddings} that guarantee monotonicity, and for which membership is easy-to-check.
Interestingly, diffusion embeddings yield a very appealing interpretation as heat diffusion dynamics where comparisons play the role of heat pumps.
A direct consequence is that categorical information (one-hot encoding embedding)
yields a monotone learning algorithm.
In particular, this class enables us to provide a positive answer
to our research question.

Finally, we evaluate the statistical performance of our 
learning algorithms through numerical experiments.\pbldelete{ on synthetic datasets.}
Our evaluations show that a linear model with good embeddings and diffusion priors outperforms the classical GBT model~\cite{DBLP:conf/aaai/FageotFHV24}, in particular with limited amount of data.

\paragraph{Related works.}
Learning preferences from comparisons has a long history, 
dating back to Thurstone~\cite{thurstone27}, 
Zermelo~\cite{zermelo1929berechnung},
and Bradley and Terry~\cite{bradley1952rank}.
To handle inconsistent judgments, 
such algorithms define a probabilistic model 
of how latent scoring of alternatives are transformed
into noisy comparisons.
Their approach was adapted by~\cite{luce1959individual} and \cite{plackett1975analysis} to model the selection of one preferred alternative out of several proposed ones; see also \cite{mcfaddenConditionalLogitAnalysis1974} and \cite[Chap. 3]{maddalaLimitedDependentQualitativeVariables1983}.

While the Bradley-Terry model considers binary-valued comparisons, various authors have proposed extensions to real-valued comparisons, e.g., ranging the interval $[-1, 1]$.
Historically, this started with the modeling of draws in games like chess \cite{davidson1970extending}.
More recently, ~\cite{DBLP:journals/corr/abs-1911-11658} 
proposed the platform \texttt{Climpact} where users are given pairs of activities,
and are asked to evaluate the comparative pollutions of two activities.
They then develop a model based on a quadratic error to turn the comparisons
into evaluations of the amounts of pollution of the individual activities.
Their model was then generalized by \cite{DBLP:conf/aaai/FageotFHV24} 
in a framework
they call \emph{Generalized Bradley-Terry} (GBT)
to turn real-valued comparisons into scores.
All these models however consider that each alternative's score
is an independent latent variable to be learned.
Thus, they fail to \emph{generalize} to non-evaluated alternatives.

Independent of user-provided comparisons, alternatives usually come with descriptive information.
\gbreplace{\cite{DBLP:journals/nca/MenkeM08,csiszar2012algorithms,zhao2016deep} proposed to model an alternative's score as a parametrized function of a vector embedding of the description of the alternative.}{A natural idea to generalize is then to model an alternative's score as a parametrized function of a vector embedding of the description of the alternative; see e.g., \cite{DBLP:journals/nca/MenkeM08,csiszar2012algorithms,zhao2016deep,duanGeneralizedModelMultidimensional2017,guIntransitivityModelMatchup2021}.}
Recently, this trick has been widely used in the context of language models~\cite{vaswani2017attention,brown2020language}, especially through algorithms like \emph{Reinforcement Learning with Human Feedback} (RLHF)~\cite{DBLP:conf/nips/ChristianoLBMLA17,DBLP:conf/nips/StiennonO0ZLVRA20}, \emph{Direct Preference Optimization} (DPO)~\cite{dpo}, or \emph{Generalized Preference Optimization} (GPO)~\cite{DBLP:conf/icml/TangGZCMRRVPP24}, to name a few.
In this paper, we restrict ourselves to \emph{linear models}, where an alternative's score is assumed to be a linear function of their embedding.
\gbedit{Application-wise, we focus on social choice applications, and rule out Supervised FineTuning applications.}
Such linear models of preferences trained from comparative judgments have previously been studied and used, e.g. by~\cite{DBLP:conf/ijcai/GuoTKOCCEDI18,DBLP:conf/aaai/NoothigattuGADR18,DBLP:journals/pacmhci/LeeKKKYCSNLPP19}.

The study of the mathematical guarantees of preference learning algorithms
has only emerged recently.
In particular, nonlinear models have been empirically shown to violate
monotonicity properties~\cite{DBLP:conf/nips/ChenMZ0ZRC24,DBLP:journals/corr/abs-2402-13228,DBLP:journals/corr/abs-2410-08847}.
While \cite{bareillesMonotonicityAIAlignment2025} proved that nonlinear models
nevertheless provide a weak monotonicity guarantee they call
\emph{local pairwise monotonicity},
they also suggest that these models are unlikely to verify 
stronger forms of monotonicity. 
Conversely, and quite remarkably, 
\cite{DBLP:conf/aaai/FageotFHV24} proved that
the GBT model guarantees monotonicity for all GBT root laws.
Our model extends GBT in several ways.
\gbdelete{Finally, in addition to AI model alignement, we note that learning preferences from comparisons with mathematical guarantees also contributes to developing more transparent social medial, as advocated recently by "prosocial media" \cite{weyl2025prosocialmedia}.
In particular, this work finds a direct application in the collaborative scoring of social media content \cite{hoang2021tournesolquestlargesecure,hoang2024solidagomodularcollaborativescoring}, and could extend the consensus-driven polling platform Pol.is \cite{small2021polis}, or the scholar recommender system~\cite{flicke2025scholarinboxpersonalizedpaper}.}

\paragraph{Paper structure.}
The rest of the paper is organized as follows.
Section~\ref{sec:monotonicity-formal-definition} introduces the formalism, formally defines monotonicity, and recalls the GBT model.
Section~\ref{sec:linear-gbt-with-diffusion-prior} defines the linear GBT model with diffusion prior,
and states our main results.
Section~\ref{sec:proof} provides the main lines of the proofs of the main results.
Section~\ref{sec:experiments} reports on our experiments.
Section~\ref{sec:conclusion} concludes.

\section{Monotonicity of Scoring Models}
\label{sec:monotonicity-formal-definition}

In this Section, we set notations, formalize the notion of monotonicity, and recall the Generalized Bradley-Terry model.

\subsection{Notations and operations on datasets}
\label{subsec:dataset-type}

Consider a set $\cA$ of $A$ alternatives. For simplicity, we let $\cA \triangleq \set{1, 2, \ldots, A}$.
The set $\setComp \subseteq \mathbb{R}$  denotes the set of admissible comparison values,
which we assume to be symmetric around zero, i.e. $r \in \setComp \iff -r \in \setComp$.
In the classical Bradley-Terry model, we have $\setComp = \set{-1, +1}$.
The generalized Bradley-Terry model allows a wider variety of possible comparison values, for instance $\setComp = [-1, 1]$, or $\setComp = \bbR$ for the uniform and gaussian root laws.
A comparison sample is defined as a triple $(a,b,r)$
where $a,b \in \cA$ are two distinct alternatives, and $r \in \setComp$.
We assume $(a, b, r)$ and $(b, a, -r)$ to be equivalent,
which we write $(a, b, r) \simeq (b, a, -r)$.
By also having $(a, b, r) \simeq (a, b, r)$ (and the relation false otherwise),
we obtain an equivalence relation.
A dataset $\bD$ is a list $\bD : [N] \to \cA^2 \times \setComp$ of comparison samples.
We write $\cD \triangleq \bigcup_{N \in \setN} (\cA^2 \times \setComp)^N$ for the set of datasets,
and $\card{\bD}$ the length of a dataset $\bD \in \cD$.
We now define four parameterized operations $\cD \to \cD$ on datasets.

\paragraph{Exchange.}
For any $n \in \setN$, $\exchange_n(\bD)$ is the dataset obtained from $\bD$ by replacing, if it exists, the $n$-th entry $(a_n, b_n, r_n)$ with $(b_n, a_n, - r_n)$
All other entries are left unchanged.
Assuming that preference learning algorithms should interpret 
these two comparison samples identically,
this operation should not affect training.

\paragraph{Shuffle.}
For any $N \in \setN$ and any permutation $\sigma$ of $[N]$, 
$\shuffle_{N, \sigma}(\bD)$ is the dataset obtained from $\bD$ by reordering its $N$ first elements according to $\sigma$.
Formally, if $\card{\bD} \geq N$, then for all $n \in [N]$ we have $\shuffle_{N, \sigma} (\bD)_n = \bD_{\sigma(n)}$.
Otherwise, $\bD$ is left unchanged.
Assuming that preference learning algorithms should be invariant to shuffling,
this operation should not affect training.

\paragraph{Append.}
For any comparison sample $(a, b, r)$,
$\append_{a,b, r} (\bD)$ is the dataset obtained from $\bD$ by appending $(a, b, r)$.
Formally, we have $\card{\append_{a,b,r} (\bD)} = \card{\bD} + 1$, and $\append_{a,b,r} (\bD)_{\card{\bD} + 1} = (a, b, r)$.
All other entries are the same as in $\bD$.
An append is said to \pbledit{definitely} favor $a'$ over $b'$ 
if it has parameters $(a, b, r) \simeq (a', b', \max \setComp)$.
Note that if $\setComp$ does not have a maximum,
then no append \pbledit{definitely} favors $a'$ over $b'$. 

\paragraph{Update.}
For any $n \in \setN$ and comparison $r \in \setComp$,
$\update_{n, r} (\bD)$ is the dataset obtained from $\bD$ by replacing the comparison of the $n$-th entry with $r$.
The update is said to favor $a$ over $b$ if either (i) $(a_n, b_n) = (a, b)$ and $r \geq r_n$, or (ii) $(a_n, b_n) = (b, a)$ and $r \leq r_n$.
In other words, it favors $a$ over $b$
if it acts on a comparison sample between $a$ and $b$,
and modifies the comparison $r$ to further favor $a$.

\subsection{Monotonicity}

\begin{definition}[Favoring $a$]
An operation $o$ \emph{favors} $a$
if $o$ is a composition of the operations
(i) $\exchange{}$, (ii) $\shuffle{}$, 
(iii) $\append{}$ that \pbledit{definitely} favor $a$ over some other alternative
and (iv) $\update{}$ that favor $a$ over some other alternative.
We write $\bD \preceq_a \bD'$
if there exists an operation $o$ that favors $a$
such that $\bD'= o(\bD)$.
\label{def:favor-a}
\end{definition}

The relation $\preceq_{a}$ is a preorder.
Indeed, $\preceq_{a}$ is reflexive: any dataset $\bD$ equals $o(\bD)$ with $o = \update{}_{n, r}$ with $n=1$, and $r=r_1$. The relation $\preceq_{a}$ is transitive: if $\bD_1 \preceq_a \bD_2$ and $\bD_2 \preceq_a \bD_3$, then there exists operations $o_1$ and $o_2$ that favor $a$, such that $\bD_1 = o_1(\bD_2)$, and $\bD_2=o_2(\bD_3)$; thus $\bD_1=o_1 \circ o_2(\bD_3)$, where $o_1\circ o_2$ is an operation that favors $a$ by Definition 1 so that $\bD_1 \preceq_a \bD_3$.
Similarly, we define the preorder $\leq_a$ over $\setR^A$ by $\theta \leq_a \theta'$ if $\theta_a \leq \theta'_a$ coordinate-wise.
We can now formally define monotonicity.

\begin{definition}[Monotonicity]\label{def:monotonicity}
    The preference learning algorithm $\ALG$ is \emph{monotone} when,
    for every alternative $a \in \cA$,
    $\ALG: (\cD, \preceq_a) \to (\mathbb{R}^A, \le_a)$
    is monotone.
    Equivalently, $\ALG$ is monotone when,
    for every alternative $a \in \cA$,
    $\bD \succeq_a \bD'$ implies $\ALG(\bD) \ge_a \ALG(\bD')$.
\end{definition}

\begin{remark}
    In the sequel, 
    all preference learning algorithms that we will consider 
    are \emph{neutral}~\cite{myerson2013fundamentals},
    i.e. they treat all alternatives symmetrically\footnote{
    Formally, $\ALG(\sigma \cdot \bD) = \sigma \cdot \ALG$
    for all permutations of $\cA$, with the action that applies pointwise
    to all apparitions of an alternative $a \in \cA$.
    }.
    For such algorithms, the monotonicity for any single $a \in \cA$ clearly implies that
    for all $a \in \cA$.
\end{remark}

\subsection{(Generalized) Bradley-Terry and monotonicity}

Here we recall the probabilistic model of GBT \cite{DBLP:conf/aaai/FageotFHV24}, slightly adapting it to our dataset formalism.\footnote{
In \cite{DBLP:conf/aaai/FageotFHV24}, the authors 
consider datasets that contain at most one comparison per pair of alternatives.
}
Following Bradley and Terry~\cite{bradley1952rank},
GBT defines a probabilistic model of comparisons given scores.
Specifically,
given two alternatives $a, b \in \cA$ having scores $\theta_a$ and $\theta_b$,
the probability of observing a value $r$ for the comparison of $a$ relative to $b$ is
\begin{equation}\label{eq:probamodel}
    p(r | \theta_{\ab}) \propto f(r) \cdot \exp(r \cdot \theta_{\ab}),
\end{equation}
where $\theta_{\ab} \eqdef \theta_a - \theta_b$ is the score difference between $a$ and $b$, and $f$ is the \emph{root law}, a probability distribution on $\setComp$ that describes comparisons when $a$ and $b$ have equal scores.

Given a dataset $\bD = (a_n,b_n,r_n)_{n \in [N]}$ of $N$ independent observations following \eqref{eq:probamodel}, and assuming a gaussian prior with zero mean and  $\sigma^2$ variance for each alternative score $\theta_a$, $a \in \cA$, 
the standard Maximum A Posteriori methodology results in the GBT estimator:
\begin{equation}
    \GBT_{f, \sigma}(\bD) \eqdef \argmin_{\theta \in \bbR^\cA} 
    \frac{1}{2\sigma^2} \sum_a \theta_a^2
    + \sum_{(a,b,r) \in \bD} \Phi_f(\theta_{\ab}) -   r \theta_{\ab}.
    \label{eq:original-gbt-loss-function}
\end{equation}
There, $\Phi_f$ is the cumulant-generating function of the root law distribution $f$: $\Phi_f(\theta) \eqdef \log \int_{\setComp} e^{r \theta} df(r)$.
As soon as $f$ has finite exponential moments, $\Phi_f$ is well-defined and convex; in particular, \eqref{eq:original-gbt-loss-function} is a strongly convex problem with a unique minimizer \cite{DBLP:conf/aaai/FageotFHV24}.

\gbreplace{Theorem 2 \cite{DBLP:conf/aaai/FageotFHV24} shows that $\GBT_{f, \sigma}$ is monotone, in a context where two elements can only be compared once.}{We recall below Theorem 2 \cite{DBLP:conf/aaai/FageotFHV24}, that guarantees monotonicity of $\GBT_{f, \sigma}$, when two elements can only be compared once.}
The forthcoming Theorem~\ref{th:monotonicity-with-good-embeddings} extends the result to situations where two elements are compared multiple times.

\begin{proposition}[Th. 2, \cite{DBLP:conf/aaai/FageotFHV24}]
    \gbreplace{For any root law $f$, and $\sigma > 0$, $\GBT_{f, \sigma}$ is monotone.}{Consider a root law $f$, a scalar $\sigma>0$, and two datasets $\bD$, $\bD'$ which contains at most one comparison between any pair $(a, b)\in\cA^2$. Then, for all $a$, $\bD \succeq_a\bD'$ implies $\GBT_{f, \sigma}(\bD) \ge_a \GBT_{f, \sigma}(\bD')$.}
\end{proposition}

Although well behaved in many aspects, the generalized Bradley-Terry model fails to perform generalization: an alternative $a$ that never appears in $\bD$ will receive a nil score $\GBT(\bD)_a = 0$.
However, in practice, alternatives may \emph{(i)} admit informative descriptions, and \emph{(ii)} have known relationships.
This should help guess the score of a yet uncompared alternative, based on the scoring of similar compared alternatives.
We introduce such a learning algorithm in the next Section.

\section{Linear GBT with Diffusion Prior}
\label{sec:linear-gbt-with-diffusion-prior}

In this Section, we introduce a class of preference learning algorithms that incorporate both user comparisons and contextual information on the compared elements,
and state our main result about
their mathematical guarantees on monotonicity.

\subsection{Learning with prior similarities}
\label{sec:learning-with-prior-similarities}

The $\GBT_{\sigma,f}$ model does not include prior knowledge
on the structure of the alternatives.
For example, when alternatives represent videos on YouTube,
the fact that two videos belong to the same channel cannot be
represented in Equation~\eqref{eq:original-gbt-loss-function}.
More generally, Equation~\eqref{eq:original-gbt-loss-function}
does not encode any prior similarities between alternatives.
Consequently, if an alternative $a$ is never compared with any other,
then, even if $a$ is similar to an alternative $b$ that has a large non-zero score,
$a$ will still be assigned a zero score.
In other words, Equation~\eqref{eq:original-gbt-loss-function}
does not allow us to generalize.

To address this issue, we generalize $\GBT_{\sigma, f}$ \eqref{eq:original-gbt-loss-function}
in two directions.
\begin{enumerate}[leftmargin=3ex]
  \item \emph{(Embeddings)} We assume that, to each alternative $a \in \cA$, corresponds an \emph{embedding} $x_a \in \bbR^D$, where $D$ is a positive integer.
    We model the score of alternative $a$ by a linear function of the embedding: \[\theta_a(\beta) \eqdef x_a^T \beta,\] for a parameter $\beta$.
    Denoting $x \in \bbR^{D \times A}$ the matrix collecting all embeddings, and $x_{\ab{}} = x_a - x_b$ for any $a, b \in\cA$, the GBT parameter $\theta\in\bbR^\cA$ is replaced by a linear function $\theta(\beta) = x^T\beta$.
    For instance, in the context of YouTube, $x_a$ could denote a one-hot encoding the content creator identity; more in Section \ref{sec:experiments}.
  \item \emph{(Similarity)} We consider a more general regularization term $\reg(\beta)$
    of the form
    \begin{equation}
      \reg(\beta)
      = \frac{1}{2\sigma^2} \sum_{d} \beta_d^2
      + \frac{1}{2} \sum_{ab} \theta_a(\beta) L_{ab} \theta_b(\beta)
    \end{equation}
    where $L$ is a Laplacian matrix, i.e. such that
    $L_{aa} = \sum_{b\ne a} |L_{ab}| \ge 0$ and $L_{ab} = L_{ba} \le 0$,
    for all $a \ne b$.
    The matrix $L$ can be thought as the Laplacian of
    a graph encoding (prior) similarities between alternatives,
    the weight $|L_{ab}|$ representing the similarity between $a$ and $b$.
    Therefore, the regularization term
    $\sum_{ab} \theta_a(\beta) L_{ab} \theta_b(\beta)
    = \frac{1}{2}\sum_{a \ne b} |L_{ab}|(\theta_a(\beta) - \theta_b(\beta))^2$
    incentivizes the model to (a priori) assign similar scores to similar alternatives.
\end{enumerate}

We can now define the class of GBT models that we will study in this paper.
\begin{definition}[Linear GBT with Diffusion Prior]
    \label{def:gbt-with-prior-similarities}
    Let $f$ be a root law, $x$ be an embedding, $\sigma > 0$ a positive constant, and $L$ a Laplacian matrix.
    The model $\GBT_{f,\sigma,x,L}$ is defined as
            $$\GBT_{f,\sigma,x,L}(\bD) \eqdef x^T \beta^*(\bD) \in \setR^A,$$
    where $\beta^* (\bD) \triangleq \argmin \cL(\cdot | \bD)$ minimizes the strongly convex loss function
    \begin{equation}
    \label{eq:gbt-datasets-loss-function}
        \cL(\beta | \bD) =
        \reg(\beta)
          + \sum_{(a,b,r) \in \bD} \Phi_f(x_{\ab}^T \beta) - r x_{\ab}^T \beta.
    \end{equation}
    For conciness, let $\theta^*(\bD) \eqdef \GBT_{f,\sigma,x,L}(\bD) = x^T \beta^*(\bD)$.
\end{definition}

Remark that the original GBT is a special case of Linear GBT with Diffusion Prior with $A=D$, $x = I$ the identity matrix, and $L = 0$.

\begin{proposition}
\label{prop:neutrality}
    Linear GBT with diffusion prior is \emph{neutral},
    i.e. invariant up to alternative relabeling.
\end{proposition}

\begin{proof}
    See Appendix~\ref{app:neutrality} for a formal statement and derivation.
\end{proof}

\subsection{Monotonicity and diffusion}
\label{sec:monotonicity-diffusion}

We now present our main result (Theorem~\ref{th:diffusion-embedding-is-good}).
We prove that for a special class of embeddings, namely \emph{diffusion embeddings},
monotonicity is guaranteed.
Diffusion embeddings take their name from 
the interplay with (super) laplacian matrices.

\begin{definition}[Super-Laplacian matrix]
   A super-Laplacian matrix
   $\Delta$ is a symmetric matrix 
   such that for all $a \ne b$,
       $\Delta_{aa} > - \sum_{b \ne a} \Delta_{ab}$
       and $\Delta_{ab} \le 0$.
\end{definition}

\begin{definition}[Diffusion embedding]
    An embedding $x$ is a diffusion embedding
    if the Gram matrices $X_\lambda = x^Tx + \lambda I$
    have super-Laplacian inverses $X_\lambda^{-1}$ for any $\lambda > 0$.
\end{definition}

Note that if $X = x^T x$ is itself invertible with super-Laplacian inverse, then it is a diffusion embedding. However, this case is restrictive since it implies that $D\geq A$.

\begin{theorem}[Monotonicity with diffusion embeddings]
    \label{th:diffusion-embedding-is-good}
    For any  root law $f$, 
    positive constant $\sigma > 0$, diffusion embedding $x$,
    and Laplacian matrix $L$,
    $\GBT_{f,\sigma,x,L}$ is monotone.
\end{theorem}

\begin{proof}
    The theorem follows directly from Proposition~\ref{prop:diffusion-is-good} and Theorem~\ref{th:monotonicity-with-good-embeddings}, which are provided and proved in Section~\ref{sec:proof}.
\end{proof}

\subsection{Example: one-hot encoding}
\label{subsec:onehotencoding}

A one-hot encoding
is possible when the alternatives can be arranged into multiple disjoint
classes. For example, if the alternatives represent videos on YouTube,
one can partition them by the YouTube channel they belong to.
In that case,
the score of an alternative $a$ is defined as $\theta_a = \gamma_{d(a)} + s^2 \cdot \alpha_a$,
where $d(a)$ is the channel of $a$.
The score $\gamma_{d(a)}$ represents the score of the channel $d(a)$,
while $\alpha_a$ represents a residual score of $a$,
and $s$ is a real constant that controls the scale 
of the residual score. 
Theorem~\ref{theo:onehotencoding} states
a one-hot encoding is an example of diffusion embedding.
We postpone the proof to Appendix \ref{app:theooneshot}.

\begin{theorem}[GBT with one-hot encoding]
\label{theo:onehotencoding}
    Let $f$ be a root law,
    $\sigma > 0$ a positive constant,
    $L$ a Laplacian matrix and $s \in \mathbb{R}$.
    Let $x : \mathbb{R}^{D \times \cA}$
    be a \emph{one-hot encoding} matrix: $x_{da} = 1$ if, and only if, $a$ belongs to $d$.
    Then, for any real number $s$, 
    $\begin{pmatrix} x & sI \end{pmatrix}^T$ is a diffusion embedding
    and the score $\GBT_{f,\sigma,x,L}$ is monotone.
\end{theorem}

\section{The proof}
\label{sec:proof}

This Section provides the mathematical analysis that builds to the proof of the main result, Theorem~\ref{th:diffusion-embedding-is-good}.
Section \ref{sec:differential-analysis} proposes a differential analysis framework for the dataset operations outlined above; Section \ref{sec:monotproofemb} then provides the proof.

\subsection{Differential analysis of dataset operations}
\label{sec:differential-analysis}

The goal of this technical section
is to connect the discrete domain of datasets
with tools from differential analysis.
Studying the monotonicity (Definition~\ref{def:monotonicity})
of $\theta^*(\bD)$ requires to compare the loss functions
for datasets that are related by a basic operation.
Given that the loss is invariant under $\exchange{}$ and $\shuffle{}$ operations
on the dataset $\bD$,
on one hand because of the specific form of the GBT loss,
and on the other because it features a sum of comparison samples of the dataset,
the same invariance holds for $\theta^*$.
Thus, to prove monotonicity,
it suffices to study what happens under $\append{}$ and $\update{}$ operations
that favor $a$ over $b$.

Because the loss function is a sum of terms indexed 
by the elements of the dataset, this relation is quite simple
\begin{align}
    \cL(\beta | \append{}_{a,b,r}(\bD)) &= \cL(\beta | \bD)
    + \Phi_f(\theta_{\ab}(\beta)) - r \theta_{\ab}(\beta) \\
    \cL(\beta | \update{}_{n,r}(\bD)) &= \cL(\beta | \bD)
     - (r - r_n) \theta_{\anbn}(\beta)
\end{align}
To enable the differential analysis of these operations,
we introduce a smooth deformation of the loss function.
\begin{definition}[Smoothed loss]
  For every $\lambda \in \mathbb{R}$,
  and every operation $o$ of the form
  $\append_{a, b, r}$ or $\update_{n, r}$, %
  we define the smoothed loss $\cL_\lambda$ by
  \begin{equation}
      \cL_\lambda (\beta | \bD, o) \triangleq 
      \cL(\beta | \bD) + 
      \lambda \cdot
      \begin{cases}
      \Phi_f(\theta_{\ab}(\beta)) - r\theta_{\ab}(\beta) &\text{if } o = \append_{a, b, r}, \\
      - (r - r_n) \cdot \theta_{\anbn} (\beta) &\text{if } o = \update_{n, r}, \\
      0 &\text{otherwise.}
      \end{cases}
  \end{equation}
  Denote also $\beta^*_\lambda(\bD, o) \triangleq \argmin \cL_\lambda(\cdot | \bD, o)$
  and $\theta^*_\lambda(\bD, o) \triangleq \theta(\beta^*_\lambda(\bD, o))$.
\end{definition}

The smoothed loss matches the loss at $\lambda \in \set{0, 1}$,
  as $\cL_0(\beta | \bD, o) = \cL(\beta | \bD)$
  and $\cL_1(\beta | \bD, o) = \cL(\beta | o(\bD))$.
We will leverage this by using
the integral expression
\begin{equation}
    \theta^*(o(\bD)) - \theta^*(\bD)
    = \int_0^1 \frac{d \theta^*_{\lambda }}{d\lambda}(\bD, o) d\lambda.
\end{equation}
When $\lambda \mapsto \theta^*_\lambda(\bD, o)$
is continuously differentiable,
this integral expression is well-defined,
and it suffices that the derivative $d\theta^*_{\lambda a}(\bD,o)/d\lambda$ be non-negative 
for the score difference at $a$ to be non-negative.
Lemma \ref{lem:derivatives-datasets-operations} states that this derivative is well defined 
and provides a formula. The proof is given in Appendix \ref{app:differential}.

\begin{lemma}
    \label{lem:derivatives-datasets-operations}
    Let $H = H(\theta | \bD)$ denote 
    the Hessian of 
    $\cE(\theta | \bD) = \sum_{(a,b,r)\in\bD} \Phi_f(\theta_{\ab})$, and $X = \sigma^2 x^Tx$ denote
    the (scaled) Gram matrix of the embedding $x$.
    Then, for any basic operation $o$ and dataset $\bD$,
    the loss function \gbreplace{$\cL(\beta | o_\lambda(\bD))$}{$\cL_\lambda ( \cdot | \bD, o)$}
    admits a unique global minimizer
    $\beta^*_\lambda(\bD,o)$,
    and the inferred score 
    $\theta^*_\lambda(\bD,o)$
    is a smooth function of $\lambda$ over $[0,\infty)$.
    Moreover,
    \begin{itemize}
      \item if $o = \update_{n, r}$ and $\bD_n \simeq (a,b,s)$,
    then the score of the alternative $a$ satisfies
    \begin{equation}
        \frac{d \theta^*_{\lambda a}}{d\lambda}(\bD,o)\bigg|_{\lambda = \mu} =
        (r - s) \cdot
        e_a^T \big( I + X(L + H) \big)^{-1} X e_{\ab{}}.
        \label{eq:update-derivative}
    \end{equation}
    \gbedit{There, $e_a$ denotes the $a$-th vector of the cartesian basis of $\bbR^A$, $\theta^*_{\lambda a}$ denotes the $a$-th coordinate of $\theta^*_\lambda$, and $e_{\ab{}} = e_a-e_b$.}
      \item if $o = \append_{a, b, r}$,
        then the score of the alternative $a$ satisfies
    \begin{equation}
        \frac{d\theta^*_{\lambda a}(\bD,o)}{d\lambda}\bigg|_{\lambda = \mu} =
        (r - \Phi'_f(\theta^*_{\ab})) \cdot
        e_a^T \big( I + X(L + H + \mu \Phi_f''(\theta_{\ab}) \cdot S^{ab})  \big)^{-1} X e_{\ab{}},
        \label{eq:append-derivative}
    \end{equation}
    where $S^{ab}\in\bbR^{A\times A}$ is the Laplacian matrix of the graph over the alternatives
    with a single edge $ab$ (with weight $1$), i.e. $S^{ab}_{aa} = S^{ab}_{bb} = 1$, $S^{ab}_{ab} = S^{ab}_{ba} = -1$, and $S^{ab}_{cd} = 0$ otherwise.
    \end{itemize}
\end{lemma}

We are interested in the sign (positive or negative) of the expressions
in Equations~\eqref{eq:update-derivative} and~\eqref{eq:append-derivative}.
First, the factors $r - s$ 
and $r - \Phi'_f(\theta^*_{\ab})$ are easy to understand.
If $o = \update_{n, r}$ favors $a$ over $b$ then $r - s \ge 0$ by definition.
If $o = \append_{a, b, r}$ favors $a$ over $b$,
then $r = \sup \setComp$ which is the supremum of $\Phi'_f$ \cite[Theorem 1]{DBLP:conf/aaai/FageotFHV24}.

Therefore, if we want to compare the scores of $a$ and $b$, the important factor to study is the matrix $(I + X(L + \tilde{H}))^{-1}X$, where $\tilde{H} = H + \mu \Phi_f''(\theta_{\ab}) \cdot S^{ab}$ if $o = \append_{a, b, r}$, or $\tilde{H} = H$ if $o = \update_{n, r}$ and $\bD_n \simeq (a, b, s)$.
We study this matrix in the next section.

\subsection{Good embeddings: a sufficient condition for monotonicity}
\label{sec:monotproofemb}

In this Section, we provide a sufficient condition on the embedding for the Linear GBT model to be monotone (Definition \ref{def:good-embeddings}, Theorem \ref{th:monotonicity-with-good-embeddings}), and provide mathematical properties of this condition (Proposition \ref{prop:assymptoticdelta}). 
Finally, we show that diffusion embeddings meet the above sufficient condition (Proposition \ref{prop:diffusion-is-good}).

\begin{definition}[Good embeddings]
    Given a Laplacian matrix $Y$, 
    an embedding $x$ is \emph{$Y$-good}
    if the Gram matrix $X =  x^Tx$
    satisfies $e_a^T (I + XY)^{-1}X e_{\ab{}} \ge 0$ for all $(ab)$.
    An embedding $x$ is \emph{good} if $x$ is $Y$-good 
    for all Laplacian matrices $Y$.
    \label{def:good-embeddings}
\end{definition}

\begin{theorem}[Monotonicity with good embeddings]
    \label{th:monotonicity-with-good-embeddings}
    For any root law $f$, positive constant $\sigma > 0$, 
    Laplacian matrix $L$, and good embedding $x$,
    $\GBT_{f,\sigma,x,L}$ is monotone.
\end{theorem}

Before going to the proof, we provide some intuition.
Note first that the Hessian $H$ of
$\cE(\theta | \bD) = \sum_{(a,b,r) \in \bD} \Phi_f(\theta_{\ab})$
is also a Laplacian matrix.
    Indeed, let $G$ be the weighted graph whose edges
    are the pairs $(ab)$ of alternatives that occur in the dataset $\bD$,
    weighted by
    $G_{ab} = N_{ab} \cdot \Phi''_f(\theta_{\ab})$
    where $N_{ab}$
    is the number of occurrences of the pair $(ab)$ in $\bD$.
    Since $\Phi_f$ is convex, these weights are nonnegative.
    Then, a direct calculation shows that $H$
    is the graph Laplacian of the weighted graph $G$:
    for $a \neq b$,
        $H_{aa} = \sum_{c \neq a} N_{ac} \Phi_f''(\theta_{ac})$ and
        $H_{ab} = - N_{ab} \Phi_f''(\theta_{\ab})$.
Therefore, given any prior Laplacian matrix $L$, the matrix $L + \tilde{H}$, where $\tilde{H}$ is defined in section \ref{sec:differential-analysis}, is the Laplacian of a graph that combines the prior similarities ($L$) with the similarities inferred from the dataset at hand ($\tilde{H}$).
This observation motivates Definition \ref{def:good-embeddings}.

\begin{proof}[Proof of Theorem \ref{th:monotonicity-with-good-embeddings}]
    By Lemma~\ref{lem:derivatives-datasets-operations}, 
    if, for every operation $o$, 
    every score function $\theta$ 
    and every dataset $\bD$,
    the inequality $e_a^T (I + X(L+\tilde{H}))^{-1}X e_{\ab{}} \ge 0$ holds,
    then the score $\theta^*(\bD)$ derived from 
    the loss function of Equation~\eqref{eq:gbt-datasets-loss-function}
    is monotone.
    This is precisely implied by $x$ being good.
\end{proof}

This result motivates a more precise understanding of good embeddings. 
In the special cases $(A, D) = (2, D)$ and $(A, D) = (A, 1)$, 
we have complete characterizations (see Appendix~\ref{app:A2D1}). 
In general, however, checking goodness is not straightforward: two embeddings $x$ and $y$ can be individually good, while their concatenation $\begin{bmatrix} x & y \end{bmatrix}^\top$ fails to be good (see Propositions \ref{prop:counterexamplea} and \ref{prop:counterexampleb}, Appendix~\ref{app:counterexamples}). 
Nonetheless, any embedding can be made $Y$-good by concatenating it with a sufficiently scaled identity. We formalize this in Appendix~\ref{app:XproofI}.

\subsection{Diffusion embeddings are good}

Finally, we show that any diffusion embedding is a good embedding. 
This uses the fact that any super-Laplacian matrix $\Delta$ satisfies
$e_a^T \Delta^{-1} e_{\ab{}} \ge 0$ for any pair $(a, b)\in\cA^2$.
This result has been proved in \cite[Lemma 1]{DBLP:conf/aaai/FageotFHV24} and we provide an alternative proof highlighting the diffusion perspective \ref{app:prooflemmasuperlaplacian}.

\begin{proposition}
    \label{prop:diffusion-is-good}
    Any diffusion embedding is a good embedding.
\end{proposition}

\begin{proof}
    If $x$ is a diffusion embedding, $\lambda > 0$, and $Y$ is an arbitrary Laplacian matrix, then the matrix $X_\lambda^{-1} + Y$,
    where $X_\lambda = x^Tx + \lambda I$, is super-Laplacian.
    Consequently,
    $e_a^T(I + X_\lambda Y)^{-1}X_\lambda e_b = e_a^T(X_\lambda^{-1} + Y)^{-1}e_b \ge 0$.
    The claim follows by taking the limit $\lambda \to 0$.
\end{proof}

\section{Experimental evaluation}
\label{sec:experiments}

In this Section, we provide a numerical exploration of the prevalence of ``goodness'' for random embeddings, and the statistical error of several preference learning models.\footnote{The code is available at \url{https://github.com/pevab/gbtlab2}, and will be made publicly after the review process. We run experiments on a personal laptop with 16GB of RAM and a 2.10 GHz processor.}
Appendix \ref{sec:exper-real-world} provides complementary experiments on real-world data.

\subsection{Probability of goodness for random embeddings}
\label{sec:probagood}
To illustrate the challenges of achieving good embeddings, we generate random i.i.d. Gaussian embedding matrices x and evaluate their quality. In Figure \ref{fig:probabilities}, we examine a single Gaussian embedding $x$ (left) and its concatenation with the identity matrix $I$ (right). Our findings indicate that the goodness of $x$ is  more likely for large values of $D/A$ and significantly diminishes when $A/D$ is large. The concatenation with the identity matrix notably enhances the goodness, aligning with Proposition~\ref{prop:assymptoticdelta}.

\begin{figure}[t]
\centering
\begin{subfigure}{0.49\textwidth}
    \includegraphics[width=\textwidth]{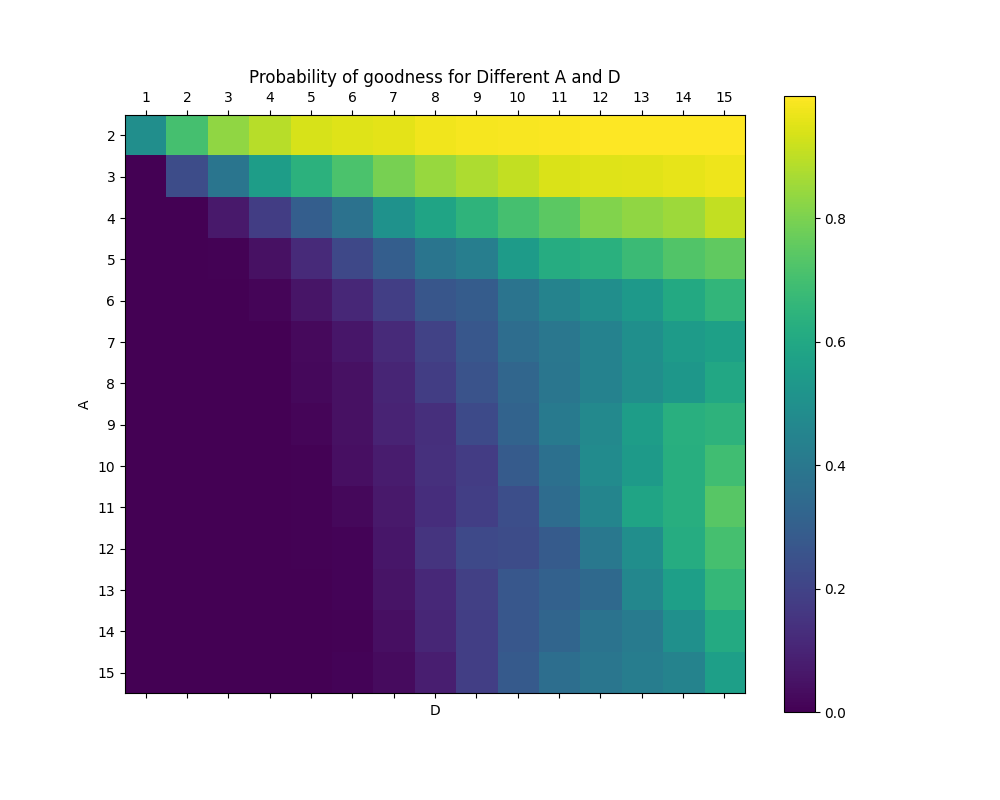}
    \caption{}
    \label{fig:Xgood}
\end{subfigure}
\hfill
\begin{subfigure}{0.49\textwidth}
    \includegraphics[width=\textwidth]{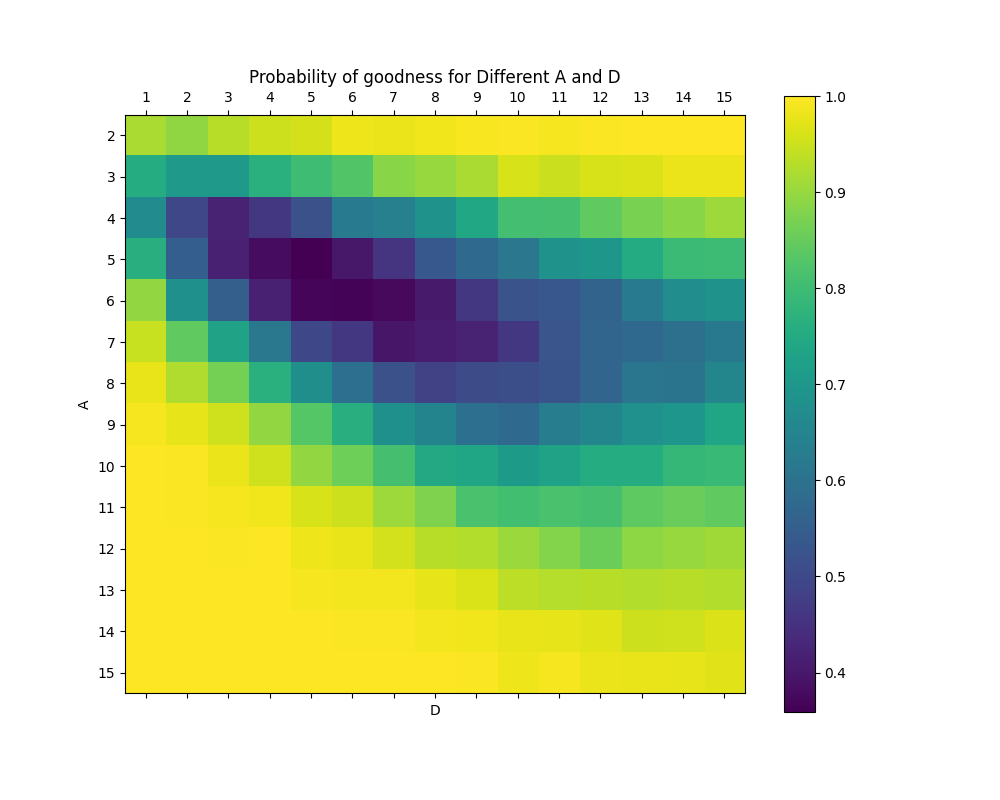}
    \caption{}
    \label{fig:[XI]good}
\end{subfigure}
\caption{
\emph{Left pane:} Probability that a Gaussian i.i.d embedding $x$ is a good embedding for $2 \leq A \leq 15$ and $1 \leq D \leq 15$.
\emph{Right pane:} As for the left pane with embedding $\begin{bmatrix}
    I & x
\end{bmatrix}^T$.
}
\label{fig:probabilities}
\vspace{-2ex}
\end{figure}

\subsection{Generative model, metric, and simulations}

\begin{figure}[t]
\centering
\begin{subfigure}{0.49\textwidth}
    \includegraphics[width=\textwidth]{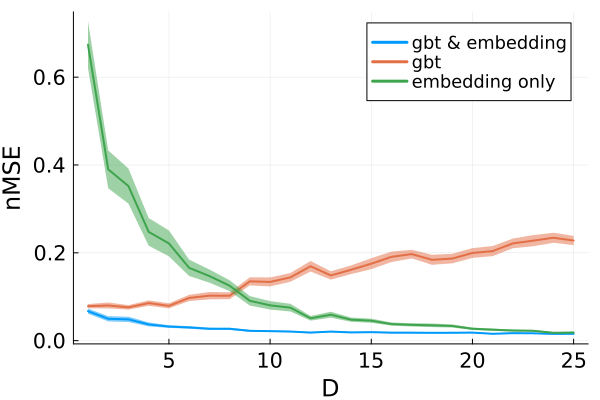}
    \caption{}
    \label{fig:MSE_D}
\end{subfigure}
\hfill
\begin{subfigure}{0.49\textwidth}
    \includegraphics[width=\textwidth]{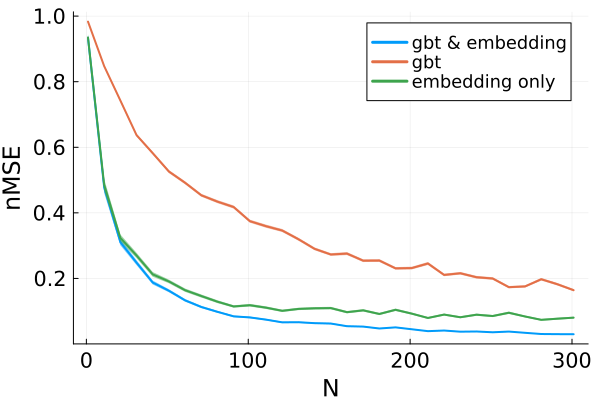}
    \caption{}
    \label{fig:MSE_N}
\end{subfigure}
\caption{
\emph{Left pane:} nMSE as a function of $D$ for $A = 25$ alternatives and $N = 500$ comparisons over $100$ seeds. Blue curve with $\begin{bmatrix} I & x \end{bmatrix}^T$ (full embedding), orange curve with embedding $I$ (classical GBT), and green curve with embedding $x$ (features only).
\emph{Right pane:} nMSE with respect to the number of comparisons $N$ for $A = 20$, $D = 10$, and $1000$ seeds. Blue curve: GBT with one-hot encoding; Orange curve: GBT.
\pbledit{Every curve is displayed with its error bar (using $1.96 \sigma / \sqrt{\text{n\_seeds}}$).}
}
\label{fig:figures}
\vspace{-2ex}
\end{figure}

For each experiment, we consider the ground-truth embedding $ x^\dagger \in \mathbb{R}^{D\times A}$,  Laplacian matrix $ L^\dagger $, constant $\sigma^\dagger$, and  root law $ f^\dagger$. The ground-truth features are generated as
$\beta^\dagger \sim \mathcal{N}(0, (\sigma^\dagger)^2 I + x^\dagger L^\dagger (x^\dagger)^T)$ and $\theta^\dagger = (x^\dagger)^T \beta^\dagger$.
We then create a dataset $ \bD : [N] \rightarrow \cA^2 \times \setComp $ by first selecting $ N $ random comparison pairs uniformly. The corresponding random comparisons $ r $ are generated using the root law $f^\dagger$ and conditionally to $\theta^\dagger$.
We shall only consider the uniform root law $ f^\dagger = \frac{1}{2} 1_{[-1 , 1]} $ and set $\sigma^\dagger = 1$.

The estimated scores are computed as 
$\theta^*(\bD) = \text{GBT}_{f,\sigma,x,L}(\bD)$, 
where the quadruplet $(f, \sigma, x, L)$ may or may not align with the ground truth. 
Since the quality of a score vector is invariant under constant shifts, 
we evaluate the error over zero-mean versions of both $\theta^\dagger$ and $\theta^*(\bD)$.
More precisely, we use Monte Carlo simulations to estimate 
the normalized mean squared error (nMSE), defined as:
$$
\text{nMSE}(f^\dagger, \sigma^\dagger, x^\dagger, L^\dagger ; f, \sigma, x, L; N) = \text{nMSE} = \mathbb{E} \left[\frac{\left\| \left(\theta^*(\bD) - \bar{\theta}^*(\bD) \right) - \left(\theta^\dagger - \bar{\theta}^\dagger \right) \right\|^2}{\|\theta^\dagger - \bar{\theta}^\dagger \|^2} \right].
$$
We then analyze how the nMSE evolves with respect to various parameters.

Figure~\ref{fig:MSE_D} shows the nMSE as a function of $D$, using data generated with $x^\dagger = \begin{bmatrix} I & \tilde{x}^\dagger \end{bmatrix}^T$ (i.i.d. Gaussian $\tilde{x}$), uniform $f^\dagger$, $L^\dagger = 0$, and $\sigma^\dagger = 1$.
We compare three models with shared parameters $(f, L, \sigma) = (f^\dagger, L^\dagger, \sigma^\dagger)$, using $x^\dagger$, $I$ (classical GBT), and $\tilde{x}^\dagger$ respectively. The embedding-based model outperforms others, combining the strengths of classical GBT for small $D$ and feature-based learning for larger $D$.

Figure~\ref{fig:MSE_N} shows the nMSE as a function of the number of comparisons $N$. Data are generated with $(f^\dagger, x^\dagger, L^\dagger, \sigma^\dagger) = \left( \frac{1}{2} 1_{[-1 , 1]}, \begin{bmatrix} I & \tilde{x}^\dagger \end{bmatrix}^T, 0 , 1 \right)$, where $\tilde{x}$ is a one-hot encoding matrix (see Section~\ref{subsec:onehotencoding}). 
The results show that one-hot encoding greatly reduces the number of comparisons needed to reach a given nMSE. This is useful in applications like YouTube score estimation \cite{hoang2024solidagomodularcollaborativescoring}, where the encoding reflects the channel and enables generalization across alternatives.

\section{Conclusion}
\label{sec:conclusion}

In this paper,
we introduced a new comparison-based preference learning model,
namely \emph{linear GBT with diffusion prior}.
This model not only generalizes to previously uncompared data using embeddings, but also potentially guarantees monotonicity, depending on the class of embeddings used.
We proved that our model is monotone for various classes
of embeddings (one-hot encodings, diffusion, and good embeddings).
To the best of our knowledge,
linear GBT with diffusion prior
is the first model that guarantees monotonicity
while being able to generalize.

Diffusion embeddings form a class
of embeddings (containing one-hot encodings) 
that yield monotonicity. 
Our proof techniques relied on an interesting
interplay between an algebraic criterion 
for monotonicity (Definition~\ref{def:good-embeddings})
and properties of (super) Laplacian matrices akin to diffusion theory.

\paragraph*{Limitations.}
While improving the understanding of preference learning with guarantees, our theory currently provides guarantees for diffusion embeddings only.
We hope to motivate more work on preference learning with guarantees, in order to build more trusworthy AI systems, with notable applications in \cite{weyl2025prosocialmedia,hoang2021tournesolquestlargesecure,hoang2024solidagomodularcollaborativescoring,small2021polis}.
Also, we caution against the use of preference learning algorithms that rely on data collected in inhumane conditions, as is mostly the case today \cite{DW2025, perrigo2023exclusive, hao2023cleaning, ghanalawsuit}.
It is unclear whether our work can positively contribute to this issue.

\subsubsection*{Acknowledgements}
The contribution of Gilles Bareilles has been funded by European Union’s Horizon Europe research and innovation programme under grant agreement No. 101070568.

\bibliographystyle{plain}
\bibliography{references.bib}

\begin{thebibliography}{10}

\bibitem{awadMoralMachineExperiment2018}
Edmond Awad, Sohan Dsouza, Richard Kim, Jonathan Schulz, Joseph Henrich, Azim
  Shariff, Jean-Fran{\c c}ois Bonnefon, and Iyad Rahwan.
\newblock The {{Moral Machine}} experiment.
\newblock {\em Nature}, 563(7729):59--64, November 2018.

\bibitem{bareillesMonotonicityAIAlignment2025}
Gilles Bareilles, Julien Fageot, L{\^e}-Nguy{\^e}n Hoang, Peva Blanchard,
  Wassim Bouaziz, S{\'e}bastien Rouault, and El-Mahdi {El-Mhamdi}.
\newblock On {{Monotonicity}} in {{AI Alignment}}, June 2025.

\bibitem{bradley1952rank}
Ralph~Allan Bradley and Milton~E Terry.
\newblock Rank analysis of incomplete block designs: I. the method of paired
  comparisons.
\newblock {\em Biometrika}, 39(3/4):324--345, 1952.

\bibitem{brown2020language}
Tom Brown, Benjamin Mann, Nick Ryder, Melanie Subbiah, Jared~D Kaplan, Prafulla
  Dhariwal, Arvind Neelakantan, Pranav Shyam, Girish Sastry, Amanda Askell,
  et~al.
\newblock Language models are few-shot learners.
\newblock {\em Advances in neural information processing systems},
  33:1877--1901, 2020.

\bibitem{DBLP:conf/nips/ChenMZ0ZRC24}
Angelica Chen, Sadhika Malladi, Lily~H. Zhang, Xinyi Chen, Qiuyi~(Richard)
  Zhang, Rajesh Ranganath, and Kyunghyun Cho.
\newblock Preference learning algorithms do not learn preference rankings.
\newblock In Amir Globersons, Lester Mackey, Danielle Belgrave, Angela Fan,
  Ulrich Paquet, Jakub~M. Tomczak, and Cheng Zhang, editors, {\em Advances in
  Neural Information Processing Systems 38: Annual Conference on Neural
  Information Processing Systems 2024, NeurIPS 2024, Vancouver, BC, Canada,
  December 10 - 15, 2024}, 2024.

\bibitem{DBLP:conf/nips/ChristianoLBMLA17}
Paul~F. Christiano, Jan Leike, Tom~B. Brown, Miljan Martic, Shane Legg, and
  Dario Amodei.
\newblock Deep reinforcement learning from human preferences.
\newblock In Isabelle Guyon, Ulrike von Luxburg, Samy Bengio, Hanna~M. Wallach,
  Rob Fergus, S.~V.~N. Vishwanathan, and Roman Garnett, editors, {\em Advances
  in Neural Information Processing Systems 30: Annual Conference on Neural
  Information Processing Systems 2017, December 4-9, 2017, Long Beach, CA,
  {USA}}, pages 4299--4307, 2017.

\bibitem{csiszar2012algorithms}
Villo Csisz{\'a}r.
\newblock Em algorithms for generalized bradley-terry models.
\newblock In {\em Annales Universitatis Scientiarum Budapestinensis de Rolando
  E{\"o}tv{\"o}s Nominatae (Sectio Computatorica)}, volume~36, pages 143--157,
  2012.

\bibitem{davidson1970extending}
Roger~R Davidson.
\newblock On extending the bradley-terry model to accommodate ties in paired
  comparison experiments.
\newblock {\em Journal of the American Statistical Association},
  65(329):317--328, 1970.

\bibitem{duanGeneralizedModelMultidimensional2017}
Jiuding Duan, Jiyi Li, Yukino Baba, and Hisashi Kashima.
\newblock A {{Generalized Model}} for {{Multidimensional Intransitivity}}.
\newblock In Jinho Kim, Kyuseok Shim, Longbing Cao, Jae-Gil Lee, Xuemin Lin,
  and Yang-Sae Moon, editors, {\em Advances in {{Knowledge Discovery}} and
  {{Data Mining}}}, pages 840--852, Cham, 2017. Springer International
  Publishing.

\bibitem{DBLP:conf/aaai/FageotFHV24}
Julien Fageot, Sadegh Farhadkhani, L{\^{e}}{-}Nguy{\^{e}}n Hoang, and Oscar
  Villemaud.
\newblock Generalized bradley-terry models for score estimation from paired
  comparisons.
\newblock In Michael~J. Wooldridge, Jennifer~G. Dy, and Sriraam Natarajan,
  editors, {\em Thirty-Eighth {AAAI} Conference on Artificial Intelligence,
  {AAAI} 2024, Thirty-Sixth Conference on Innovative Applications of Artificial
  Intelligence, {IAAI} 2024, Fourteenth Symposium on Educational Advances in
  Artificial Intelligence, {EAAI} 2014, February 20-27, 2024, Vancouver,
  Canada}, pages 20379--20386. {AAAI} Press, 2024.

\bibitem{flicke2025scholarinboxpersonalizedpaper}
Markus Flicke, Glenn Angrabeit, Madhav Iyengar, Vitalii Protsenko, Illia
  Shakun, Jovan Cicvaric, Bora Kargi, Haoyu He, Lukas Schuler, Lewin Scholz,
  Kavyanjali Agnihotri, Yong Cao, and Andreas Geiger.
\newblock Scholar inbox: Personalized paper recommendations for scientists,
  2025.

\bibitem{guIntransitivityModelMatchup2021}
Yan Gu, Jiuding Duan, and Hisashi Kashima.
\newblock An {{Intransitivity Model}} for {{Matchup}} and {{Pairwise
  Comparison}}.
\newblock In {\em 2020 25th {{International Conference}} on {{Pattern
  Recognition}} ({{ICPR}})}, pages 692--698, January 2021.

\bibitem{DBLP:conf/ijcai/GuoTKOCCEDI18}
Yuan Guo, Peng Tian, Jayashree Kalpathy{-}Cramer, Susan Ostmo, J.~Peter
  Campbell, Michael~F. Chiang, Deniz Erdogmus, Jennifer~G. Dy, and Stratis
  Ioannidis.
\newblock Experimental design under the bradley-terry model.
\newblock In J{\'{e}}r{\^{o}}me Lang, editor, {\em Proceedings of the
  Twenty-Seventh International Joint Conference on Artificial Intelligence,
  {IJCAI} 2018, July 13-19, 2018, Stockholm, Sweden}, pages 2198--2204.
  ijcai.org, 2018.

\bibitem{ghanalawsuit}
Rachel Hall and Claire Wilmot.
\newblock Meta faces ghana lawsuits over impact of extreme content on
  moderators.
\newblock {\em {T}he {G}uardian}, 2025.

\bibitem{hao2023cleaning}
Karen Hao and Deepa Seetharaman.
\newblock Cleaning up chatgpt takes heavy toll on human workers.
\newblock {\em Wall Street Journal}, 24, 2023.

\bibitem{hoang2021tournesolquestlargesecure}
L{\^{e}}{-}Nguy{\^{e}}n Hoang, Louis Faucon, Aidan Jungo, Sergei Volodin, Dalia
  Papuc, Orfeas Liossatos, Ben Crulis, Mariame Tighanimine, Isabela Constantin,
  Anastasiia Kucherenko, Alexandre Maurer, Felix Grimberg, Vlad Nitu, Chris
  Vossen, S{\'{e}}bastien Rouault, and El{-}Mahdi El{-}Mhamdi.
\newblock Tournesol: {A} quest for a large, secure and trustworthy database of
  reliable human judgments.
\newblock {\em CoRR}, abs/2107.07334, 2021.

\bibitem{hoang2024solidagomodularcollaborativescoring}
Lê~Nguyên Hoang, Romain Beylerian, Bérangère Colbois, Julien Fageot, Louis
  Faucon, Aidan Jungo, Alain~Le Noac'h, Adrien Matissart, and Oscar Villemaud.
\newblock Solidago: A modular collaborative scoring pipeline, 2024.

\bibitem{DW2025}
Stephanie H{ö}ppner.
\newblock Africa's content moderators want compensation for job trauma.
\newblock {\em Deutsche Welle}, 2025.

\bibitem{DBLP:journals/corr/abs-1911-11658}
Victor Kristof, Valentin Quelquejay{-}Lecl{\`{e}}re, Robin Zbinden, Lucas
  Maystre, Matthias Grossglauser, and Patrick Thiran.
\newblock A user study of perceived carbon footprint.
\newblock {\em CoRR}, abs/1911.11658, 2019.

\bibitem{DBLP:journals/pacmhci/LeeKKKYCSNLPP19}
Min~Kyung Lee, Daniel Kusbit, Anson Kahng, Ji~Tae Kim, Xinran Yuan, Allissa
  Chan, Daniel See, Ritesh Noothigattu, Siheon Lee, Alexandros Psomas, and
  Ariel~D. Procaccia.
\newblock Webuildai: Participatory framework for algorithmic governance.
\newblock {\em Proc. {ACM} Hum. Comput. Interact.}, 3({CSCW}):181:1--181:35,
  2019.

\bibitem{luce1959individual}
R~Duncan Luce et~al.
\newblock {\em Individual choice behavior}, volume~4.
\newblock Wiley New York, 1959.

\bibitem{maddalaLimitedDependentQualitativeVariables1983}
G.~S. Maddala.
\newblock {\em Limited-{{Dependent}} and {{Qualitative Variables}} in
  {{Econometrics}}}.
\newblock Econometric {{Society Monographs}}. Cambridge University Press,
  Cambridge, 1983.

\bibitem{mcfaddenConditionalLogitAnalysis1974}
Daniel McFadden.
\newblock Conditional logit analysis of qualitative choice behavior.
\newblock pages 105--142, 1974.

\bibitem{DBLP:journals/nca/MenkeM08}
Joshua~E. Menke and Tony~R. Martinez.
\newblock A bradley-terry artificial neural network model for individual
  ratings in group competitions.
\newblock {\em Neural Comput. Appl.}, 17(2):175--186, 2008.

\bibitem{myerson2013fundamentals}
Roger~B Myerson et~al.
\newblock Fundamentals of social choice theory.
\newblock {\em Quarterly Journal of Political Science}, 8(3):305--337, 2013.

\bibitem{DBLP:conf/aaai/NoothigattuGADR18}
Ritesh Noothigattu, Snehalkumar (Neil)~S. Gaikwad, Edmond Awad, Sohan Dsouza,
  Iyad Rahwan, Pradeep Ravikumar, and Ariel~D. Procaccia.
\newblock A voting-based system for ethical decision making.
\newblock In Sheila~A. McIlraith and Kilian~Q. Weinberger, editors, {\em
  Proceedings of the Thirty-Second {AAAI} Conference on Artificial
  Intelligence, (AAAI-18), the 30th innovative Applications of Artificial
  Intelligence (IAAI-18), and the 8th {AAAI} Symposium on Educational Advances
  in Artificial Intelligence (EAAI-18), New Orleans, Louisiana, USA, February
  2-7, 2018}, pages 1587--1594. {AAAI} Press, 2018.

\bibitem{DBLP:conf/nips/NoothigattuPP20}
Ritesh Noothigattu, Dominik Peters, and Ariel~D. Procaccia.
\newblock Axioms for learning from pairwise comparisons.
\newblock In Hugo Larochelle, Marc'Aurelio Ranzato, Raia Hadsell,
  Maria{-}Florina Balcan, and Hsuan{-}Tien Lin, editors, {\em Advances in
  Neural Information Processing Systems 33: Annual Conference on Neural
  Information Processing Systems 2020, NeurIPS 2020, December 6-12, 2020,
  virtual}, 2020.

\bibitem{DBLP:journals/corr/abs-2402-13228}
Arka Pal, Deep Karkhanis, Samuel Dooley, Manley Roberts, Siddartha Naidu, and
  Colin White.
\newblock Smaug: Fixing failure modes of preference optimisation with
  dpo-positive.
\newblock {\em CoRR}, abs/2402.13228, 2024.

\bibitem{perrigo2023exclusive}
Billy Perrigo.
\newblock Openai used kenyan workers on less than \$2 per hour to make chatgpt
  less toxic.
\newblock {\em Time Magazine}, 18:2023, 2023.

\bibitem{plackett1975analysis}
Robin~L Plackett.
\newblock The analysis of permutations.
\newblock {\em Journal of the Royal Statistical Society Series C: Applied
  Statistics}, 24(2):193--202, 1975.

\bibitem{rafailovDirectPreferenceOptimization2024}
Rafael Rafailov, Archit Sharma, Eric Mitchell, Stefano Ermon, Christopher~D.
  Manning, and Chelsea Finn.
\newblock Direct {{Preference Optimization}}: {{Your Language Model}} is
  {{Secretly}} a {{Reward Model}}, July 2024.

\bibitem{dpo}
Rafael Rafailov, Archit Sharma, Eric Mitchell, Christopher~D Manning, Stefano
  Ermon, and Chelsea Finn.
\newblock Direct preference optimization: Your language model is secretly a
  reward model.
\newblock In A.~Oh, T.~Naumann, A.~Globerson, K.~Saenko, M.~Hardt, and
  S.~Levine, editors, {\em Advances in Neural Information Processing Systems},
  volume~36, pages 53728--53741. Curran Associates, Inc., 2023.

\bibitem{DBLP:journals/corr/abs-2410-08847}
Noam Razin, Sadhika Malladi, Adithya Bhaskar, Danqi Chen, Sanjeev Arora, and
  Boris Hanin.
\newblock Unintentional unalignment: Likelihood displacement in direct
  preference optimization.
\newblock {\em CoRR}, abs/2410.08847, 2024.

\bibitem{small2021polis}
Christopher Small, Michael Bjorkegren, Timo Erkkil{\"a}, Lynette Shaw, and
  Colin Megill.
\newblock Polis: Scaling deliberation by mapping high dimensional opinion
  spaces.
\newblock {\em Recerca: revista de pensament i an{\`a}lisi}, 26(2), 2021.

\bibitem{DBLP:conf/nips/StiennonO0ZLVRA20}
Nisan Stiennon, Long Ouyang, Jeffrey Wu, Daniel~M. Ziegler, Ryan Lowe, Chelsea
  Voss, Alec Radford, Dario Amodei, and Paul~F. Christiano.
\newblock Learning to summarize with human feedback.
\newblock In Hugo Larochelle, Marc'Aurelio Ranzato, Raia Hadsell,
  Maria{-}Florina Balcan, and Hsuan{-}Tien Lin, editors, {\em Advances in
  Neural Information Processing Systems 33: Annual Conference on Neural
  Information Processing Systems 2020, NeurIPS 2020, December 6-12, 2020,
  virtual}, 2020.

\bibitem{DBLP:conf/icml/TangGZCMRRVPP24}
Yunhao Tang, Zhaohan~Daniel Guo, Zeyu Zheng, Daniele Calandriello, R{\'{e}}mi
  Munos, Mark Rowland, Pierre~Harvey Richemond, Michal Valko,
  Bernardo~{\'{A}}vila Pires, and Bilal Piot.
\newblock Generalized preference optimization: {A} unified approach to offline
  alignment.
\newblock In {\em Forty-first International Conference on Machine Learning,
  {ICML} 2024, Vienna, Austria, July 21-27, 2024}. OpenReview.net, 2024.

\bibitem{thurstone27}
Louis~Leon Thurstone.
\newblock A law of comparative judgment.
\newblock {\em Psychological Review}, 34(4):273--286, 1927.

\bibitem{vaswani2017attention}
Ashish Vaswani, Noam Shazeer, Niki Parmar, Jakob Uszkoreit, Llion Jones,
  Aidan~N Gomez, {\L}ukasz Kaiser, and Illia Polosukhin.
\newblock Attention is all you need.
\newblock {\em Advances in neural information processing systems}, 30, 2017.

\bibitem{weyl2025prosocialmedia}
E.~Glen Weyl, Luke Thorburn, Emillie de~Keulenaar, Jacob Mchangama, Divya
  Siddarth, and Audrey Tang.
\newblock Prosocial media, 2025.

\bibitem{zermelo1929berechnung}
Ernst Zermelo.
\newblock Die berechnung der turnier-ergebnisse als ein maximumproblem der
  wahrscheinlichkeitsrechnung.
\newblock {\em Mathematische Zeitschrift}, 29(1):436--460, 1929.

\bibitem{zhao2016deep}
Piplong Zhao, Ou~Wu, Liyuan Guo, Weiming Hu, and Jinfeng Yang.
\newblock Deep learning-based learning to rank with ties for image re-ranking.
\newblock In {\em 2016 IEEE International Conference on Digital Signal
  Processing (DSP)}, pages 452--456. IEEE, 2016.

\end{thebibliography}

\newpage
\appendix

\section{Experiments on real world data}
\label{sec:exper-real-world}

In this Section, we complete the experiments on synthetic data (Section \ref{sec:experiments}) with experiments on real-world data \cite{hoang2021tournesolquestlargesecure}.
More precisely, we provide numerical evidence for the fact that including a (diffusion) embeddings improves performance.
The code to reproduce experiments is available at \url{https://github.com/pevab/gbtlab2}.

\subsection{Experimental set-up}
The real-world data contains comparisons between Youtube videos made by various users, from the Tournesol platform \cite{hoang2024solidagomodularcollaborativescoring}.
We selected a subset $\bD$ of $1000$ comparisons from a single user, the one that has most comparisons.
Every comparison is a tuple $(a,b,r)$ where $a,b$ are video identifiers, and $r$ is the comparison value $r$, originally an integer between $-10$ and $10$, which we rescale to fit in $[-1,1]$.

In addition, we associate for each video $a$, the YouTube channel $c$ it belongs to. We describe this relation using a one-hot encoding matrix $\chi \in \mathbb{R}^{D \times N}$: $\chi_{ca} = 1$ if the video $a$ belongs to the channel $c$,
or $\chi_{ca} = 0$ otherwise.
We obtain an embedding $x \in \mathbb{R}^{(D + N)\times N}$ by concatenating $\chi$ with $\lambda I$, that is, $x = \begin{pmatrix} \chi \\ I_N \end{pmatrix}$.

We compare two models: \emph{(i)} $\GBT_{f, 1, x, 0}$ (which uses embeddings), 
and \emph{(ii)} $\GBT_{f, 1, I_N, 0}$ (the original $\GBT$, which does not use embeddings).
Both models have the uniform distribution in $[-1, 1]$ as a root law, $f(r) = \tfrac{1}{2} 1_{[-1, 1]}(r)$, and the same Gaussian prior. We do not use any Laplacian regularization.

After training, each model $M$ computes, given a pair $(a,b)$ of video identifiers, the expected comparison value defined as  
\begin{equation}
    M(a,b) = \int r \cdot f(r) e^{r (\theta^*_a - \theta^*_b)} dr = \Phi'_f(\theta^*_a - \theta^*_b)
\end{equation}
where $\Phi_f$ is the cumulant-generating function of the root law, 
and $\theta^*$ are the scores learned.
Given a validation dataset $\bD_{\text{val}}$, the validation risk of $M$ is given by
\begin{equation}
    \frac{1}{|\bD_{\text{val}}|} \sum_{(a,b,r) \in \bD_{\text{val}}}  (M(a,b) - r)^2
\end{equation}

\subsection{Results}

Figure~\ref{fig:real_world_comparison} reports the empirical risks of the two models, using a $10$-fold cross validation scheme over the dataset $\bD$ (1000 comparisons).
We observe that the average validation risk of the model with embeddings is $8.40 \cdot 10^{-3}$, while that of the model without embeddings is $10.1 \cdot 10^{-3}$.
Hence, including the YouTube channel embeddings reduces the risk by $17\%$ on average.

\begin{figure}
    \centering
    \includegraphics[width=0.66\linewidth]{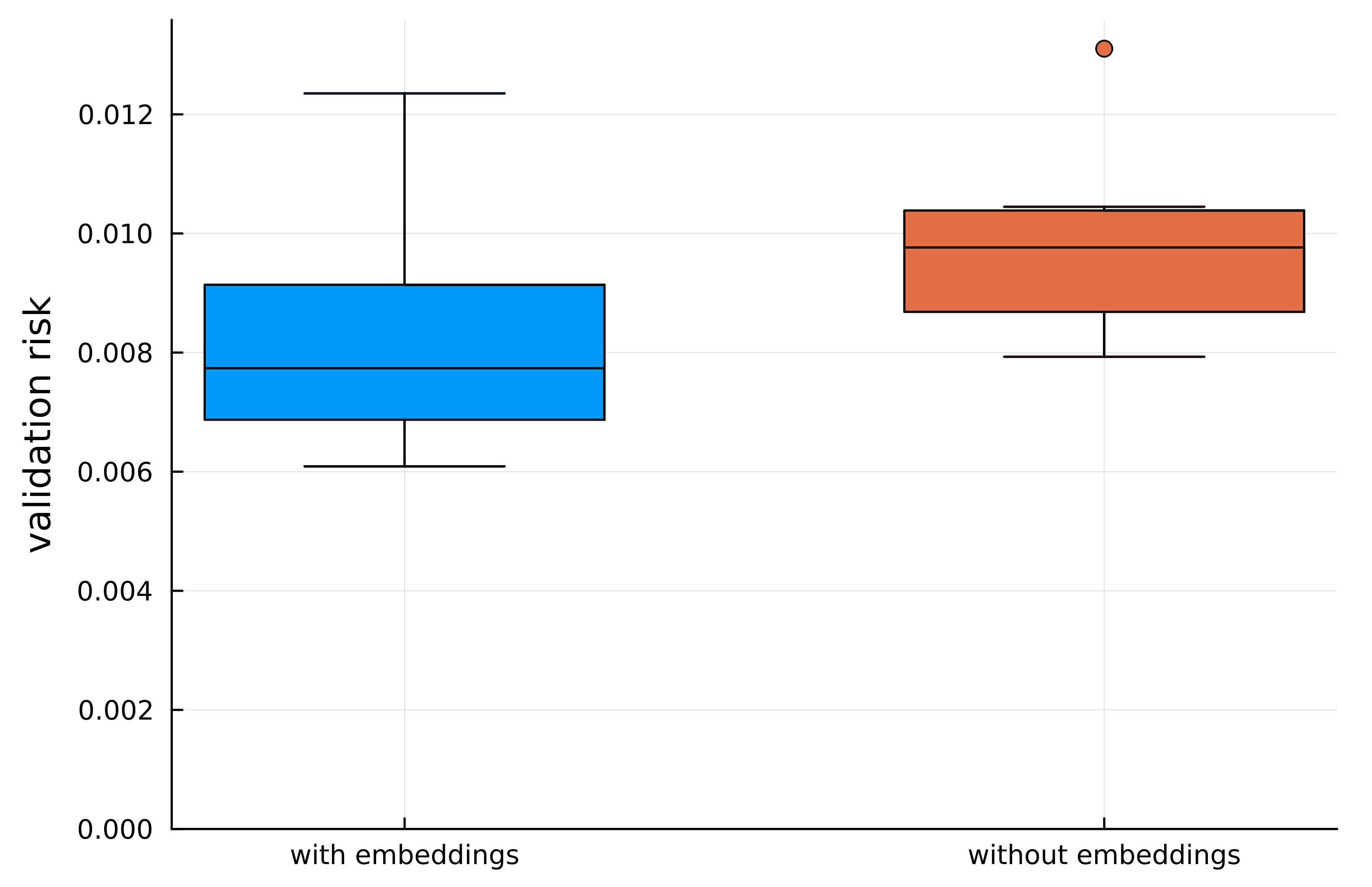}
    \caption{Validation risks of $\GBT_{f,1,x,0}$ (with embeddings, on the left) and $\GBT_{f,1,I_N,0}$ (without embeddings, on the right), estimated using $10$-fold cross validation. The box plots report the minimum, 1st quartile, median, 3rd quartile and maximum. Outliers are also shown.}
    \label{fig:real_world_comparison}
\end{figure}

\section{Proof of neutrality (Proposition~\ref{prop:neutrality})}
\label{app:neutrality}

To formalize neutrality, 
we must define how a permutation $\tau$ of the alternatives 
acts on the inputs and outputs of Linear GBT with Diffusion Prior.
We define the actions as follow:
\begin{align}
    (\tau \cdot \bD)_n &= \tau \cdot \bD_n, \\
    \tau \cdot (a, b, r) &= (\tau(a), \tau(b), r), \\
    (\tau \cdot \theta)_a &= \theta_{\tau(a)}, \\
    (\tau \cdot x)_a &= x_{\tau(a)}, \\
    (\tau \cdot L)_{ab} &= L_{\tau(a) \tau(b)}.
\end{align}
Neutrality is formalized by the equality
\begin{align}
    \forall \tau \mathsep \tau^{-1} 
        \cdot GBT_{f, \sigma, \tau \cdot x, \tau \cdot L} 
        \circ \tau 
    = GBT_{f, \sigma, x, L}.
\end{align}
We indeed have
\begin{align}
    &\left( \tau^{-1} \cdot GBT_{f, \sigma, \tau \cdot x, \tau \cdot L} \circ \tau (\bD) \right)_a
    = \left( GBT_{f, \sigma, \tau \cdot x, \tau \cdot L} \left( \tau \cdot \bD \right) \right)_{\tau^{-1}(a)} \\
    &= (\tau \cdot x)^T_{\tau^{-1}(a)} \beta_{f, \sigma, \tau \cdot x, \tau \cdot L}^*(\tau \cdot \bD)
    = x_a^T \beta_{f, \sigma, \tau \cdot x, \tau \cdot L}^*(\tau \cdot \bD) \\
    &= x^T_a \argmin_{\beta} 
        \frac{1}{2 \sigma^2} \norm{\beta}{2}^2
        + \frac{1}{2} \sum_{ab} (x^T_{\tau(a)} \beta) L_{\tau(a) \tau(b)} (x^T_{\tau(b)} \beta) \nonumber \\
        &\qquad 
        + \sum_{(a, b, r) \in \bD} \left(
            \Phi_f((\tau \cdot x)_{\tau^{-1}(a) \tau^{-1}(b)}^T \beta)
            - r (\tau \cdot x)_{\tau^{-1}(a) \tau^{-1}(b)}^T \beta
        \right) \\
    &= x^T_a \argmin_{\beta} 
        \frac{1}{2 \sigma^2} \norm{\beta}{2}^2
        + \frac{1}{2} \sum_{a'b'} (x^T_{a'} \beta) L_{a' b'} (x^T_{b'} \beta) 
        + \sum_{(a, b, r) \in \bD} \left(
            \Phi_f(x_{\ab}^T \beta)
            - r x_{\ab}^T \beta
        \right) \\
    &= x_a^T \beta_{f, \sigma, x, L}^*(\bD)
    = \left( GBT_{f, \sigma, x, L} (\bD) \right)_a.
\end{align}

\section{Proof of Lemma \ref{lem:derivatives-datasets-operations}}
\label{app:differential}
    We consider the case of $o = \append_{a, b, r}$.
    The case $o = \update_{n, r}$ with $\bD_n\simeq (a, b, s)$ is proved similarly.
    The loss functions $\cL_\lambda(\beta | \bD, o)$
    and $\cL(\beta | \bD)$
    differ by \gbreplace{$\Phi_f(\theta_{\ab}) - r \theta_{\ab}$}{$\lambda (\Phi_f(\theta_{\ab}(\beta)) - r \theta_{\ab}(\beta)$)}.
    The term $\theta_{\ab}$ being linear in $\beta$, its Hessian is zero.
    On the other hand, the Hessian of $\Phi_f(\theta_{\ab})$
    is simply the Laplacian $\Phi_f''(\theta_{\ab}) \cdot S^{ab}$ of the graph with a single edge $ab$ with weight
    $\Phi_f''(\theta_{\ab})$.
    Writing
    \begin{equation}
        H^\lambda = H + \lambda \Phi_f''(\theta_{\ab}) \cdot S^{ab},
    \end{equation}
    we obtain the Hessian of the loss
    \begin{equation}
        H_\beta \cL_\lambda(\beta | \bD, o) = x (L + H^\lambda )x^T + \frac{1}{\sigma^2}I \succeq \frac{1}{\sigma^2}I.
    \end{equation}
    Therefore, the loss $\cL_\lambda(\beta | \bD, o)$ 
    is strictly convex, and admits a global minimizer $\beta^*_\lambda(\bD,o)$.
    To simplify notations in this proof, we write $\beta^*(\lambda)$.

    We want to use the implicit function theorem
    to analyze how $\beta^*(\lambda)$ varies with $\lambda$.
    For that, consider the gradient of the loss 
    $\cL_\lambda(\beta | \bD, o)$
    \begin{equation}
        F(\beta, \lambda) = \nabla_\beta \cL(\beta | o_\lambda(\bD))
        = \nabla_\beta \cL(\beta | \bD) 
        + \lambda \cdot (\Phi'_f(\theta_{\ab}) - r) x e_{\ab}
    \end{equation}
    with domain $\mathbb{R}^D \times \mathbb{R}$
    and codomain $\mathbb{R}^D$.
    
    Fix some $\mu \in \mathbb{R}$.
    The Jacobian of $F(\beta,\lambda)$ with respect to $\beta$,
    evaluated at $\lambda = \mu$,
    is given by
    \begin{equation}
        J_\beta F(\beta, \mu) = \frac{1}{\sigma^2}I + x(L + H^\mu)x^T
    \end{equation}
    which is invertible.
    Moreover, since $\beta^*(\mu)$ is a minimizer, 
    we have $F(\beta^*(\mu), \mu) = 0$.

    Hence, the implicit function theorem
    states that there exists an open neighborhood $U$
    of $\mu$, and a smooth function $\gamma : U \to \mathbb{R}^D$
    such that
    \begin{gather}
        \gamma(\mu) = \beta^*(\mu) \\
        \forall \lambda \in U, F(\gamma(\lambda), \lambda) = 0
    \end{gather}
    The latter equality implies that $\gamma(\lambda) = \beta^*(\lambda)$
    for all $\lambda \in U$.
    Consequently, $\beta^*(\lambda)$ and $\theta^*(\lambda) = x^T \beta^*(\lambda)$ depend smoothly on $\lambda$.
    In addition, the implicit function theorem also 
    gives an expression for the Jacobian $J_\lambda \beta^*$, evaluated at $\mu$
    \begin{align}
        J_\lambda \beta^*(\mu)
        &=
        - (J_\beta F)^{-1} J_\lambda F \\
        &= (r - \Phi'_f(\theta^*_{\ab})) \cdot
        \left(\frac{1}{\sigma^2}I + x(L + H^\mu)x^T\right)^{-1}  x e_{\ab}
    \end{align}
    Finally, note that
    \begin{align}
        \frac{d \theta^*}{d\lambda}(\mu)
        &= J_\beta \theta^*(\beta^*(\mu)) \cdot J_\lambda \beta^*(\mu) \\
        &= (r - \Phi'_f(\theta^*_{\ab}))
        \cdot  x^T \left(\frac{1}{\sigma^2}I + x (L + H^\mu) x^T\right)^{-1} x e_{\ab}
    \end{align}

    Let $M = L + H^\mu$, and $X = \sigma^2 x^Tx$.
    Using Woodbury's identity $(I + UV)^{-1} = I - U(I + VU)^{-1}V$,
    we derive
    \begin{align}
        \left(\frac{1}{\sigma^2}I + x(L + H^\mu)x^T\right)^{-1}
        &= \sigma^2 (I + \sigma^2 xMx^T)^{-1}  \\
        &= \sigma^2I - \sigma^4 xM(I + \sigma^2 x^T x M)^{-1} x^T \\
        x^T\left(\frac{1}{\sigma^2}I + x(L + H)x^T\right)^{-1}x
        &= X - XM(I + XM)^{-1}X   \\
        &= (I - XM(I + XM)^{-1})X \\
        &= (I + XM)^{-1} X 
        \label{eq:woodbury-chaining}
    \end{align}
    Thus,
    \begin{equation}
        \frac{d \theta^*}{d\lambda}(\mu)
        = (r - \Phi'_f(\theta^*_{\ab}))
        \cdot (I + \sigma^2 x^\top x (L + H^\mu)) \sigma^2 x^\top x
        e_{\ab}.
    \end{equation}

    The result follows.

\section{Good Embeddings for $A = 2$ or $D=1$}
\label{app:A2D1}

 We can fully characterize goodness in lowest dimensional regime, either with only two alternatives or with an embedding on a single feature.

\begin{definition}
    We say that a matrix $M$ is max-diagonally dominant if $M_{aa} \geq M_{ab}$ for any $(ab)$.
\end{definition}

\begin{proposition}
\label{prop:strictmonotA=2}
    An embedding with Gram matrix 
    $X = \begin{bmatrix}
        a & c \\
        c & b
    \end{bmatrix}$ 
    is monotonicity proof if and only if 
    $- \sqrt{ab} \leq c \leq \min(a,b)$.
\end{proposition}

\begin{proof}
We first observe that, since $x^T x$ is positive semidefinite,  $a,b \geq 0$ and $det(x^T x) = ab - c^2 \geq 0$. This implies in particular that $c \geq - \sqrt{ab}$, which is the desired lower bound.

Two-dimensional Laplacian matrices are of the form 
$Y = \delta \left[\begin{matrix}
    1 & -1 \\
    -1 & 1 
\end{matrix}\right]$ 
with $\delta > 0$, which can be chosen as being equal to $1$ without loss of generality. We then have that $I + X Y = \left[\begin{matrix}
    1 + a - c & c - a \\
    c - b & 1 + b - c 
\end{matrix}\right]$. After simplification, its determinant is $det(I + XY) = 1 + a + b - 2c$. This quantity is strictly positive for $ c < \frac{1 + a + b}{2}$, which is always the case for as $c^2 \leq ab$. Hence, the matrix is invertible and we have, after computation,
\begin{equation}
    M = (I +XY)^{-1} X = \frac{1}{1 + a + b - 2c} \left[\begin{matrix}
    a + ab - c^2  & c + ab  - c^2  \\
    c + ab - c^2 & b + ab - c^2.
\end{matrix}\right]
\end{equation}
Then, $M$ is max-diagonally dominant if and only if $a \geq c$ et $b \geq c$, as expected. Finally, this shows that the embedding $x$ is $Y$-good for any $Y$, hence good.
\end{proof}

We now focus on the case $A=2$, for which we obtain a complete characterization of the goodness.

\begin{proposition} \label{prop:strongmonoto2alternatives}
Consider the GBT model with embedding $x = [x_a, x_b] \in \mathbb{R}^{D \times 2}$ and root law $f$. Then, the model is good if and only if, for any $(a,b)\in\cA^2$, $x_a = x_b$ or $x_a^T x_b \leq \min( \| x_a \|^2 , ||x_b \|^2)$. The latter is equivalent to
    \begin{equation}
    \alpha(x_a ,x_b) \leq \alpha_0( \| x_a \|, \|x_b \|) =  \arccos \left(\min\left( \frac{\|x_a\|}{\|x_b\|}, \frac{\|x_b\|}{\|x_a\|} \right) \right) \in [0,\pi/2]
    \end{equation}
    where $\alpha(x_a ,x_b) = \frac{x_a^T x_b}{ \|x_a \| \| x_b \| }$ is the angle between $x_a$ and $x_b$.
\end{proposition}

 \begin{proof}
    We apply Proposition \ref{prop:strictmonotA=2} to $a = x_a^2$, $b = x_b^2$, and $c = x_a x_b$. Then, the goodness is equivalent to $x_a x_b \leq \min(x_a^2, x_b^2)$. This relation is true for $x_a = x_b$ and $x_a x_b <0$. Otherwise, we have $0 < x_a , x_b$. The relations $x_a x_b \leq x_a^2$ and $x_a x_b \leq x_b^2$ implies that $x_b \leq x_a$ and $x_a \leq x_b$ respectively, leading to a contradiction. Finally, the goodness is equivalent to $x_a = x_b$ or $x_a$ and $x_b$ have different signs. 

     For the $D$ dimensional case, the same observation holds with $a = \| x_a \|^2$, $b = \| x_b \|^2$, and $c = x_a^T x_b$. Then, Proposition \ref{prop:strictmonotA=2} implies that the goodness is equ bivalent to $x_a^T x_b \leq \min(\| x_a \|^2, \| x_b \|^2)$. The angular characterization follows easily. 
 \end{proof}

 We shall see that the goodness is very restricted for $D=1$.

     \begin{proposition} \label{prop:partialthetaD1}
     We consider the GBT model with embedding $x = [x_a]_{a\in \mathcal{A}} \in \mathbb{R}^{1 \times A}$, root law $f$, and Laplacian matrix $L = 0$.
    The model is good if and only if
     \begin{equation} \label{eq:xstronglymonotonousD1}
        x =  [u, \ldots ,u, -v, \ldots, -v, 0, \ldots , 0]^T P = [u 1_{A_1}, - v 1_{A_2}, 0_{A_3}]^T P
     \end{equation}
     for some $u,v > 0$ and $P$ a permutation matrix.
    \end{proposition}
    
\begin{proof}
     The goodness is equivalent to the fact that, for any $(ab)$, $(x_a - x_b)x_a \geq 0$ and $(x_b - x_a) x_b \geq 0$. This is equivalent to $x_a x_b \leq min (x_a^2, x_b^2)$, i.e. $x_a = x_b$ or $x_a x_b \leq 0$. This means that any $x_a > 0$ should have a common value $u>0$ and any $x_b < 0$ should have a common value $- v < 0$. Hence, $x_a$ can take only the values $u, -v$ and $0$. Permuting the values, we obtain \eqref{eq:xstronglymonotonousD1}.
 \end{proof}

\section{Counterexamples for Monotonicity}
\label{app:counterexamples}

\begin{proposition}
\label{prop:counterexamplea}
    There exists good embeddings $x_1$ and $x_2$ such that $x=\begin{bmatrix}
        x_1 &
        x_2
    \end{bmatrix}^T$ is not good.
\end{proposition}

\begin{proof}
     For $A=3$ and a Gaussian root law ($Y = 3 I - J$), we consider $x_1$ and $x_2$ such that
$$
X_1 = \begin{bmatrix}
0 & 0 & 0 \\
0 & 1 & 1 \\
0 & 1 & 1
\end{bmatrix} \quad
 \text{and} \quad
X_2 = \begin{bmatrix}
1 & 1 & 0 \\
1 & 1 & 0 \\
0 & 0 & 0
\end{bmatrix}$$
Then, $X = x^T x = x_1^T x_1 + x_2^T x_2 = X_1 + X_2 = \begin{bmatrix}
1 & 1 & 0 \\
1 & 2 & 1 \\
0 & 1 & 1
\end{bmatrix}$.
Then, $x_1$ and $x_2$ are good embedding (e.g. remarking that they are $J$-blocs matrices and using Theorem \ref{theo:onehotencoding}). The matrix $M = (I + XY)^{-1} X$ is given by
$$M =\frac{1}{8} \begin{bmatrix}
            {3} & {4} & 1 \\
            {4} & 8 & 4 \\
            1 & 4 & 3 
\end{bmatrix}$$
and verifies $M_{12} > M_{11}$. This contradicts Definition \ref{def:good-embeddings} and $x$ is not $Y$-proof for $Y = 3 I - J$, therefore not good. 
\end{proof}

\begin{proposition}
\label{prop:counterexampleb}
    There exists one-hot encodings $x_1$ and $x_2$ such that $x=\begin{bmatrix}
        x_1 &
        x_2
    \end{bmatrix}^T$ is not a good embedding.
\end{proposition}

\begin{proof}
For $A=5$, let $x_1$ and $x_2$ be one-hot encoding with Gram matrices
    \[
X_1 =
\begin{bmatrix}
1 & 1 & 1 & 0 & 0 \\
1 & 1 & 1 & 0 & 0 \\
1 & 1 & 1 & 0 & 0 \\
0 & 0 & 0 & 1 & 0 \\
0 & 0 & 0 & 0 & 1
\end{bmatrix} \quad \text{and} \quad
X_2 =
\begin{bmatrix}
1 & 0 & 0 & 0 & 0 \\
0 & 1 & 0 & 0 & 0 \\
0 & 0 & 1 & 1 & 1 \\
0 & 0 & 1 & 1 & 1 \\
0 & 0 & 1 & 1 & 1
\end{bmatrix}.
\]
Then, $x_1$ and $x_2$ are good embeddings according to  Theorem \ref{theo:onehotencoding}. The Gram matrix of the concatenated embedding $x = \begin{bmatrix}
    x_1 \\ x_2
\end{bmatrix}$ is
$$
X = X_1 + X_2 =
\begin{bmatrix}
2 & 1 & 1 & 0 & 0 \\
1 & 2 & 1 & 0 & 0 \\
1 & 1 & 2 & 1 & 1 \\
0 & 0 & 1 & 2 & 1 \\
0 & 0 & 1 & 1 & 2
\end{bmatrix} 
$$
is not a $Y$-good embedding for the Laplacian matrix
$$Y = \begin{pmatrix} 2 & -1 & 0 & 0 & -1 \\ -1 & 4 & -1 & -1 & -1 \\ 0 & -1 & 1 & 0 & 0 \\ 0 & -1 & 0 & 1 & 0 \\ -1 & -1 & 0 & 0 & 2 \end{pmatrix}.$$
We indeed have that
$$M = (I + XY)^{-1} X \approx \begin{pmatrix} 0.96 & 0.70 & 0.85 & 0.52 & 0.67 \\ 0.70 & 0.90 & 0.95 & 0.67 & 0.68 \\ 0.85 & 0.95 & 1.48 & 0.83 & 0.84 \\ 0.52 & 0.67 & 0.83 & 1.08 & 0.56 \\ 0.67 & 0.68 & 0.84 & 0.56 & 0.93 \end{pmatrix}$$
is such that $M_{23} > M_{22}$.
\end{proof}

\section{Results of Section \ref{sec:monotproofemb}}
\label{app:XproofI}

We now show that any embedding can be made $Y$-good by appending a sufficiently dominant identity component. This result guarantees that, asymptotically, adding uncorrelated features improves embedding monotonicity, regardless of the original embedding structure. This can be made in relation with Figure \ref{fig:probabilities} for which we compare i.i.d. Gaussian $x$ with it's concatenation with $I$.
\begin{proposition} \label{prop:assymptoticdelta}
    For any embedding $x$ and any Laplacian matrix $Y$, the embedding $x_\lambda = \begin{bmatrix}
        I &
        x / \lambda
    \end{bmatrix}^T$ is $Y$-good for any 
    $\lambda > 3 \sqrt{A} \| x^T x \| / \mathrm{DiagDom}(\sigma^2Y)$
    where $\mathrm{DiagDom}(Y) = \min_{(ab)} (I + Y)^{-1}_{aa} -  (I + Y)^{-1}_{ab} > 0$.
\end{proposition}

\begin{proof}[Proof of Proposition \ref{prop:assymptoticdelta}]
We set $M(X,Y) = (I + XY)^{-1}X$.
The Frobenius norm is such that $\| X + Y \| \leq \| X \| + \| Y\|$ and $\| X  Y \| \leq \| X \| \| Y\|$.
Assuming that $\|X\| < 1$ and for $Z = (I + X)^{-1} - (I - X)$, we have 
\begin{equation} \label{eq:normZ}
    \| Z \| \leq \| X \|^
2 \sum_{k\geq 0} \| X \|^k = \frac{\| X \|^2}{1 - \| X \|}.
\end{equation} 
The matrix $I + X / \lambda$ is positive definite, hence invertible and we have
\begin{align}
    M(I + X /\lambda,Y) = (I + (I+X/\lambda)Y)^{-1} (I+X/\lambda) = ( (I+X/\lambda)^{-1} + Y)^{-1}.
\end{align}
Then, there exist matrices $Z$ and $W$, that we will control later on, such that
\begin{align} \label{eq:MXdelta}
    M(I + X /\lambda,Y) &=  ( I - X/\lambda+ Z + Y)^{-1} \\
    &= (I + Y)^{-1} (I + (Z - X / \lambda) (I +Y)^{-1})^{-1} \\
    &= (I + Y)^{-1} (I + (Z - X / \lambda) (I +Y)^{-1} + W).
\end{align}
From now, we assume that $\lambda > 3\sqrt{A}\|X\|$. In particular, using that $h : x \mapsto \frac{x}{1 - x}$ is increasing over $(0,1)$ we have that $ \frac{  \| X \| / \lambda}{1 -   \| X \| / \lambda } \leq h(\sqrt{A}/3) \leq h(1/3) = \frac12$.
According to \eqref{eq:normZ} applied to $X/\lambda$, we therefore have 
\begin{equation} %
    \| Z \| \leq   \frac{\| X / \lambda \|^2}{1 - \| X / \lambda \|} \leq \frac{\| X\|}{2 \lambda}.
\end{equation} 
We also have that
\begin{equation} %
    \| X / \lambda - Z \| \leq \frac{\| X\|}{ \lambda} + \frac{\| X\|}{2 \lambda} = \frac{3 \| X\|}{2 \lambda}.
\end{equation} 
Note that $\frac{3 \| X\|}{2 \lambda} \leq \frac12$ since $\lambda > 3 \sqrt{A} \|X\|$ and we can apply \eqref{eq:normZ} to evaluate
\begin{align} \label{eq:normW}
    \| W  \| &\leq \frac{\| (X/\lambda - Z) (I + Y)^{-1}\|^2}{1 - \| (X/\lambda - Z)  (I + Y)^{-1}\| } 
    \leq \frac{ A \|  (X/\lambda - Z)\|^2}{1 - \sqrt{A}\| X/\lambda - Z\|} = 
    \sqrt{A} \| X/\lambda - Z \| h(\sqrt{A}\| X/\lambda - Z\|) \\
    &\leq \sqrt{A} \frac{3 \| X\|}{2 \lambda} h\left( \frac{3 \sqrt{A}\| X\|}{2 \lambda}\right) \leq
     \frac{3\sqrt{A}  \| X\|}{2 \lambda} h\left( 1/2 \right)  = \frac{3\sqrt{A}  \| X\|}{2 \lambda}
\end{align} 
where we used $\| (X/\lambda - Z) (I + Y)^{-1}\| \leq \| X/\lambda - Z \| \|(I + Y)^{-1}\| \leq \sqrt{A} \| X/\lambda - Z \|$ (since $(I + Y)^{-1}$ has eigenvalues smaller than 1) and $h(1/2) = 1$. \\

Starting again with \eqref{eq:MXdelta}, we then have 
\begin{align}
    \| M(I + X /\lambda,Y) - (I + Y)^{-1} \|_\infty
    &\leq     \| M(I + X /\lambda,Y) - (I + Y)^{-1} \| \\
   & = \| (Z - X / \lambda) (I +Y)^{-1} + W \|
    \leq \frac{3 \sqrt{A} \| X\|}{2 \lambda} + \frac{3\sqrt{A}  \| X\|}{2 \lambda} \\
    &= \frac{3\sqrt{A}  \| X\|}{\lambda}.
\end{align}
Let $\mathrm{DiagDom} (Y) = \min_{(ab)} \left((I + Y)^{-1}_{aa} - (I + Y)^{-1}_{ab}\right)$ which is strictly positive since $(I + Y)^{-1}$ is strictly max-diagonally dominant 
 \cite{DBLP:conf/aaai/FageotFHV24}. We therefore have that, for any $\lambda > 3\sqrt{A}\|X\| /\mathrm{DiagDom} (Y)$, $$\| M(I + X /\lambda,Y) - (I + Y)^{-1} \|_\infty \leq \mathrm{DiagDom} (Y) $$ and therefore $M(I + X /\lambda,Y)$ is max-diagonally dominant, as expected.
\end{proof}

\section{Inverse of super-Laplacian}
\label{app:prooflemmasuperlaplacian}   

\begin{proposition}
Let $\Delta$ be a super-Laplacian matrix, then for any nodes $a \neq b$ in $\cA$, we have
\[
e_a^T \Delta^{-1} e_{\ab} = (\Delta^{-1})_{aa} - (\Delta^{-1})_{ab} \geq 0.
\]
\end{proposition}

\begin{proof}
    We prove the result by interpreting the coefficients of $\Delta^{-1}$ as a probability
    of some sample path of a discrete-time Markov process on the alternatives.
    Let $D$ be the diagonal of $\Delta$, and $P$ the matrix defined by
    \begin{equation}
        \Delta = D(I - P).
    \end{equation}
    The matrix $P$ is a row-sub-stochastic matrix.
    Explicitly,
    \begin{align}
        P_{aa} &= 0, & P_{ab} &= \frac{|\Delta_{ab}|}{\Delta_{aa}}, & \sum_{b \in \cA} P_{ab} &< 1.
    \end{align}
    Let $\bullet$ be an extra symbol and $\cA_\bullet = \cA \sqcup \{\bullet\}$. Let $\kappa_a = \Delta_{aa} - \sum_{b \ne a} |\Delta_{ab}| > 0$.
    We define a Markovian random walk on $\cA_\bullet$ by the transition matrix $T$:
    \begin{align}
        T(b|a) &= P_{ab} = \frac{|\Delta_{ab}|}{\Delta_{aa}}, & T(\bullet|a) &= 1 - \sum_{b \in \cA} P_{ab} = \frac{\kappa_a}{\Delta_{aa}}, & T(a|\bullet) &= 0, & T(\bullet|\bullet) &= 1.
    \end{align}
    Intuitively, this Markovian process walks over the alternatives according to the weights $|\Delta_{ab}|$, and at each step has a non-zero probability to end in the cemetery $\bullet$.

    Now,
    \begin{align}
        \Delta^{-1}_{ba} 
        &= \big((1 - P)^{-1} D^{-1}\big)_{ba}, & \Delta^{-1}_{ba} &= \sum_{n \ge 0} \big(P^n D^{-1}\big)_{ba}, & \Delta^{-1}_{ba} &= \sum P_{b a_1} \dots P_{a_{n-1}a} \frac{1}{\Delta_{aa}}.
    \end{align}
    Therefore,
    \begin{equation}
        \Delta^{-1}_{ba}\kappa_a = \sum T(a_1|b) \dots T(a|a_{n-1}) T(\bullet|a).
    \end{equation}
    In other words, $\Delta^{-1}_{ba}\kappa_a$ is the probability that $a$ is the last alternative to be visited before the random walk is killed, given that it started at $b$. Notice that the alternative $a$ may be visited multiple times in those paths. Actually, any path that starts at $b$ and visits $a$ before being killed can be decomposed as the gluing of a path that starts at $b$ and reaches $a$, followed by a path that starts at $a$ and eventually revisits $a$ as its last step before being killed. Therefore,
    \begin{equation}
        \Delta^{-1}_{ba}\kappa_a \le \Delta^{-1}_{aa} \kappa_a.
    \end{equation}
    Since $\kappa_a > 0$, and since $\Delta$ is symmetric, we finally obtain
    \begin{equation}
        \Delta^{-1}_{aa} \ge \Delta^{-1}_{ab}.
    \end{equation}    
\end{proof}

\section{Proof of Theorem \ref{theo:onehotencoding}}
\label{app:theooneshot}

\begin{proof}
    Fix an arbitrary $\lambda > 0$,
    and let $\mu = s^2 + \lambda$.
    There exists a permutation matrix $P$,
    and an integer partition $A_1 + \dots + A_k = A$,
    such that
    \begin{align}
        X &\triangleq \begin{pmatrix}
           x^T & sI 
        \end{pmatrix}
        \begin{pmatrix}
           x \\ sI 
        \end{pmatrix}
        + \lambda I \\
        &= x^Tx + \mu I  \\
        &= P \cdot
        \Big(
            \mu I + \text{block\_diagonal}(J_{A_1},\dots,J_{A_k}) 
        \Big) \cdot P^{-1} \\
        &= P \cdot
            \text{block\_diagonal}(
            \mu I_{A_1} + J_{A_1},
            \dots,
            \mu I_{A_k} + J_{A_k}
            ) \cdot P^{-1}
    \end{align}
    where every $J_{A_i}$ is a matrix of size $A_i \times A_i$
    with all its entries set to $1$.
    Up to renaming the alternatives, we can assume,
    without loss of generality, that $P = I$.

    We notice that
    the all-one matrix $J$, say of size $A$, satisfies
    $J^2 = A J$, and
    \begin{equation}
        (\mu I + J)\frac{1}{\mu }\Big(I - \frac{1}{A+\mu}J\Big) = 1
    \end{equation}
    Therefore, $X$ is invertible and
    \begin{equation}
        X^{-1} = \text{block\_diagonal}
        \Big(
        \frac{1}{\mu}\big(I_{A_1} - \frac{1}{A_1 + \mu}J_{A_1}\big),
        \dots,
        \frac{1}{\mu}\big(I_{A_k} - \frac{1}{A_k + \mu}J_{A_k}\big)
        \Big)
    \end{equation}
    Thus, $X^{-1}$ is super-Laplacian.
    This proves that $\begin{pmatrix} x & sI \end{pmatrix}^T$ is a diffusion embedding.
    The monotonicity of $\GBT_{f,\sigma,x,L}$ follows 
    from Theorem~\ref{th:diffusion-embedding-is-good}.
\end{proof}

\newpage
\section*{NeurIPS Paper Checklist}

\begin{enumerate}

\item {\bf Claims}
    \item[] Question: Do the main claims made in the abstract and introduction accurately reflect the paper's contributions and scope?
    \item[] Answer: \answerYes{} %
    \item[] Justification: The introduction summarizes the claims in the ``Contributions'' paragraph, with references to the mathematical statements in the rest of the paper.

\item {\bf Limitations}
    \item[] Question: Does the paper discuss the limitations of the work performed by the authors?
    \item[] Answer: \answerYes{} %
    \item[] Justification: We discuss limitations of this work in a dedicated paragraph of the Conclusion section.

\item {\bf Theory Assumptions and Proofs}
    \item[] Question: For each theoretical result, does the paper provide the full set of assumptions and a complete (and correct) proof?
    \item[] Answer: \answerYes{} %
    \item[] Justification: Each mathematical statement has a proof, either in the main body or the appendix. Pointers to the location of proofs are provided.

    \item {\bf Experimental Result Reproducibility}
    \item[] Question: Does the paper fully disclose all the information needed to reproduce the main experimental results of the paper to the extent that it affects the main claims and/or conclusions of the paper (regardless of whether the code and data are provided or not)?
    \item[] Answer: \answerYes{} %
    \item[] Justification: 
    All the code and instructions necessary to reproduce the experiments are provided in the Supplementary Material. Section 5 also provides a detailed explanation of each step of the numerical experiments, along with the computing environment.

\item {\bf Open access to data and code}
    \item[] Question: Does the paper provide open access to the data and code, with sufficient instructions to faithfully reproduce the main experimental results, as described in supplemental material?
    \item[] Answer: \answerYes{} %
    \item[] Justification: 
    Again, all the code and instructions necessary to reproduce the experiments are provided in the Supplementary Material.
    The code will be uploaded on GitHub if the paper is accepted.

\item {\bf Experimental Setting/Details}
    \item[] Question: Does the paper specify all the training and test details (e.g., data splits, hyperparameters, how they were chosen, type of optimizer, etc.) necessary to understand the results?
    \item[] Answer: \answerYes{} %
    \item[] Justification: A description of the setup used to provide Fig. 1 and 2 is provided in Section 5.

\item {\bf Experiment Statistical Significance}
    \item[] Question: Does the paper report error bars suitably and correctly defined or other appropriate information about the statistical significance of the experiments?
    \item[] Answer: \answerNo{} %
    \item[] Justification: The plots provide the normalized mean squared error of the estimators, averaged over a high number of repetitions (100 repetitions). As such, it is unclear what additional insight variance metrics would bring to the experiments.

\item {\bf Experiments Compute Resources}
    \item[] Question: For each experiment, does the paper provide sufficient information on the computer resources (type of compute workers, memory, time of execution) needed to reproduce the experiments?
    \item[] Answer: \answerYes{} %
    \item[] Justification: The experiments require little computing power, details on RAM, CPU are provided in Section 5.

\item {\bf Code Of Ethics}
    \item[] Question: Does the research conducted in the paper conform, in every respect, with the NeurIPS Code of Ethics \url{https://neurips.cc/public/EthicsGuidelines}?
    \item[] Answer: \answerYes{} %
    \item[] Justification: 
    \item[] Guidelines:
    We reviewed the NeurIPS Code of Ethics, and found that none of the problematic cases in it conform with the numerical experiments.
    \begin{itemize}
        \item The answer NA means that the authors have not reviewed the NeurIPS Code of Ethics.
        \item If the authors answer No, they should explain the special circumstances that require a deviation from the Code of Ethics.
        \item The authors should make sure to preserve anonymity (e.g., if there is a special consideration due to laws or regulations in their jurisdiction).
    \end{itemize}

\item {\bf Broader Impacts}
    \item[] Question: Does the paper discuss both potential positive societal impacts and negative societal impacts of the work performed?
    \item[] Answer: \answerYes{} %
    \item[] Justification:
    We mention broader impacts both at the end of the Related Works section of the Introduction, and in the limitations paragraph of the Conclusion
    

\item {\bf Safeguards}
    \item[] Question: Does the paper describe safeguards that have been put in place for responsible release of data or models that have a high risk for misuse (e.g., pretrained language models, image generators, or scraped datasets)?
    \item[] Answer: \answerNA{} %
    \item[] Justification: : The paper does not release any new data or new models.

\item {\bf Licenses for existing assets}
    \item[] Question: Are the creators or original owners of assets (e.g., code, data, models), used in the paper, properly credited and are the license and terms of use explicitly mentioned and properly respected?
    \item[] Answer: \answerNA{} %
    \item[] Justification: The paper does not rely on existing assets, all experiments are based on synthetic data.

\item {\bf New Assets}
    \item[] Question: Are new assets introduced in the paper well documented and is the documentation provided alongside the assets?
    \item[] Answer: \answerNA{} %
    \item[] Justification: The paper does not introduce any new assets.

\item {\bf Crowdsourcing and Research with Human Subjects}
    \item[] Question: For crowdsourcing experiments and research with human subjects, does the paper include the full text of instructions given to participants and screenshots, if applicable, as well as details about compensation (if any)?
    \item[] Answer: \answerNA{} %
    \item[] Justification: The paper does not conduct crowdsourcing and research with human subjects.

\item {\bf Institutional Review Board (IRB) Approvals or Equivalent for Research with Human Subjects}
    \item[] Question: Does the paper describe potential risks incurred by study participants, whether such risks were disclosed to the subjects, and whether Institutional Review Board (IRB) approvals (or an equivalent approval/review based on the requirements of your country or institution) were obtained?
    \item[] Answer: \answerNA{} %
    \item[] Justification:  Again, the paper does not conduct research with human subjects.

\item {\bf Declaration of LLM usage}
    \item[] Question: Does the paper describe the usage of LLMs if it is an important, original, or non-standard component of the core methods in this research? Note that if the LLM is used only for writing, editing, or formatting purposes and does not impact the core methodology, scientific rigorousness, or originality of the research, declaration is not required.
    \item[] Answer: \answerNA{} %
    \item[] Justification: The core method development in this research does not involve LLMs, nor does the writing of the paper.

\end{enumerate}

\end{document}